%% file: ex_article.tex
\documentclass[onefignum,onetabnum]{siamart171218}

\input{ex_shared}

\ifpdf
\hypersetup{
  pdftitle={Efficient subsampling-rate selection for practical multilevel Markov chain Monte Carlo},
  pdfauthor={Lise Guilliams, Pieter Vanmechelen, and Giovanni Samaey}
}
\fi

\begin{document}

\maketitle

\begin{abstract}
 The Multilevel Markov chain Monte Carlo algorithm accelerates Bayesian inversion
 by exploiting correlations between discretized quantities of interest across a
 hierarchy of model resolutions. A key practical parameter is the subsampling rate
 used to generate coarse-level proposals for the finer-level chains. Existing
 practice sets this rate equal to an integrated autocorrelation time as a proxy
 for proposal decorrelation, but this choice can be overly conservative and lead
 to substantial computational overhead. In this work we treat finite subsampling as an
 error-budget question. We derive a cost--error
 criterion for selecting a computationally efficient subsampling rate using
 estimates of sample variances, integrated autocorrelation times, and per-level
 costs. The criterion balances multilevel sampling errors across levels rather
 than enforcing approximate independence of the coarse proposals.
 We clarify the theoretical role of independence through an ideal reset-reference
 kernel, for which independent coarse proposals may be replaced by proposals
 generated from a reversible coarse-level kernel. We derive a Poisson-equation representation of the perturbation bias resulting from finite subsampling. Under a state-averaged lower-level forgetting assumption, this yields a conditional geometric perturbation-bias bound; in computation, the corresponding projected diagnostic is used as an empirical error-budget check.
 Numerical experiments on a highly autocorrelated elasticity beam benchmark show
 substantial reductions in computational cost compared with IAT-based
 subsampling. A second Darcy-flow inverse problem gives supporting evidence of
 qualitatively similar finite-subsampling behaviour in another elliptic
 benchmark.
\end{abstract}



\section{Introduction}
\label{ch:introduction}
With the increasing availability of data, a common challenge across various
engineering fields is the estimation of unknown parameters within a given model
such that it accurately represents the observed data. Let \(\mathcal{F}\) denote
the forward model, which maps the parameter \(\theta\) to the observed data
\(u\) as
\begin{equation}
    u = \mathcal{F}(\theta)+\gamma
\end{equation}
where \(\gamma\) denotes the observational noise, which is assumed to be
independent and identically distributed Gaussian noise with zero mean and
diagonal covariance~\cite{baecher_reliability_2003}. In the setting of this
paper the main forward model is an elliptic partial differential equation (PDE)
approximated on a finite element grid, and the unknown coefficient field is
represented by a finite truncation of an underlying random field.

For the notation used below, Bayesian inference is based on the posterior
distribution
\[
    P(\theta|u) \propto \mathcal{L}(u|\theta)\pi(\theta),
\]
where \(\pi(\theta)\) denotes the prior and \(\mathcal{L}(u\mid\theta)\) the
likelihood~\cite{cotter_probabilistic_2015, smith_uncertainty_2024}. For
Gaussian observational noise this likelihood is given by
\cite{dodwell_hierarchical_2015}
\begin{equation}
    \mathcal{L}(u|\theta) =
    \exp \left(\frac{-||u-\mathcal{F}(\theta)||^2}{2\sigma^2_{F}} \right),
    \label{ingevulde likelihood}
\end{equation}
where \(\sigma_F\) is the fidelity parameter. Evaluating the likelihood
therefore requires solving the forward problem for a given \(\theta\), which is
typically the dominant computational cost in MCMC methods~\cite{law_data_2015,
tarantola_inverse_2005, hastings1970, metropolis1953}.

The Multilevel Monte Carlo (MLMC) method~\cite{giles_multilevel_2008} addresses
this trade-off between accuracy and computational cost by introducing a
hierarchy of discretization levels \(\ell=0,\ldots,L\). The finest level
provides the most accurate approximation but at the highest computational cost,
whereas the coarser levels are less expensive but also less accurate. This leads
to the following telescopic decomposition for the MLMC estimator:
\begin{equation}
    \begin{aligned}
        \widehat{Q}_{L}
        &= \frac{1}{N_0} \sum_{n=b_0+1}^{b_0+N_0} Q_0(\theta_0^n)
        + \sum_{\ell=1}^{L} \frac{1}{N_\ell} \sum_{n=b_\ell+1}^{b_\ell+N_\ell} 
        \big(Q_\ell(\theta_\ell^n) - Q_{\ell-1}(\theta_{\ell-1}^n)\big).
    \end{aligned}
    \label{MLMC}
\end{equation}
Here \(b_\ell\) denotes the number of burn-in samples discarded at level
\(\ell\). By aggregating estimators across all levels, MLMC methods can attain
the accuracy of the finest discretization at a reduced overall computational
cost.

The principles of the MLMC framework have been combined with the
Metropolis--Hastings algorithm to form the Multilevel Markov chain Monte Carlo
(MLMCMC) algorithm~\cite{dodwell_hierarchical_2015}. The MLMCMC method employs a
hierarchical estimator for the QoI based on the telescoping sum in
Equation~\eqref{MLMC}. Samples at each level are generated using an MH-type
accept--reject procedure, where proposals on finer levels, \(\ell>0\), are
obtained by subsampling the chain at level \(\ell-1\). The frequency with which
coarse-level samples are selected as proposals is the subsampling rate
\(\tau_\ell\). Existing recommendations suggest selecting this rate equal to the
integrated autocorrelation time (IAT) of the coarse-level Markov chain, using
the IAT as a proxy for decorrelation of the proposed samples~\cite{dodwell_hierarchical_2015}.
A central point of this work is that this IAT-based decorrelation rule can be
overly conservative for the cost--error objective of MLMCMC.
This leads to the main algorithmic contribution of this paper: a criterion for
selecting the subsampling rate in practical MLMCMC algorithms. Instead of
choosing the subsampling rate \(\tau_\ell\) only to approximate independent
coarse-level proposals, the proposed criterion balances the statistical benefit
of more strongly decorrelated proposals against the additional lower-level cost
incurred in generating them.

The independence assumption is not essential in an ideal reset-reference
construction: a reversible lower-level kernel is sufficient for the standard
fine-level Metropolis--Hastings acceptance probability
\cite{metropolis1953,hastings1970,tierney_1994}. This argument applies only to
the reset-reference transition and does not establish exactness of the practical
continued-chain implementation used here, where the lower-level chain may
continue after fine-level rejection. We therefore treat finite subsampling in
the practical chain as a perturbation and assess its magnitude empirically,
using a signed Poisson diagnostic that depends on implementation choices such as
the binning of the augmented state. We therefore use this diagnostic as a
practical check on the error budget in the chosen finite-dimensional projection,
rather than as a general certificate that finite-subsampling bias is absent.
This distinguishes the present approach from exact delayed-acceptance
constructions such as MLDA, which modify the implemented transition itself to
preserve invariance~\cite{christen_markov_2005,
lykkegaard_multilevel_2023}; related exact multilevel chains can also be
obtained under independent-MH proposal
structures~\cite{madrigalcianci2021analysisclassmultilevelmarkov}.

Building on the MLMCMC estimator, multilevel acceptance construction, and
IAT-based subsampling choice of~\cite{dodwell_hierarchical_2015}, as well as
the likelihood scaling and plane-stress benchmark of~\cite{vanmechelen_multilevel_2025},
this paper makes four contributions:
\begin{itemize}
    \item we derive a cost-based criterion for choosing the subsampling rate in
    practical MLMCMC;
    \item we propose an adaptive implementation that estimates variances, IATs,
    and level costs to compute the subsampling rate during the simulation;
    \item we clarify the role of reversibility through a finite-\(\tau\)
    reset-reference calculation and state the caveat for the practical
    continued-chain implementation;
    \item we use Poisson-equation perturbation diagnostics to assess the effect
    of reduced subsampling rates empirically.
\end{itemize}
The main numerical benchmark is an elasticity beam inverse problem. A
Darcy-flow inverse problem provides a supporting elliptic benchmark.

This paper is structured as follows. Section~\ref{ch:model setup} introduces the
common inverse-problem, random-field, KL, and multilevel notation. 
Section~\ref{ch:MLMCMC chapt} summarizes MLMCMC and the standard subsampling
choice. Section~\ref{ch:deriving efficient sub} studies finite-subsampling
effects through the reset-reference calculation, the cost and accuracy
diagnostics, and the signed Poisson diagnostic.
Section~\ref{sec:subsamp_derivation} derives and implements the proposed
subsampling strategy. Section~\ref{sec:subsamp_numerics} reports the numerical
experiments, and Section~\ref{ch:conclusion} contains the conclusion.


\section{Model setup}
\label{ch:model setup}
This section elaborates on the model ingredients that are common to both numerical examples. Both examples are Bayesian inverse problems for parameterized elliptic PDEs. The forward map \(\mathcal{F}\) introduced in Section~\ref{ch:introduction} is obtained by solving the PDE for a spatial coefficient determined by the parameter \(\theta\), and the posterior is defined through the Gaussian observation model in Equation~\eqref{ingevulde likelihood}. In practice, instead of observing the full PDE solution, we assume access only to a discrete set of observations from which the spatial profile of \(\theta\) is inferred. The quantity of interest is a scalar functional of the PDE solution or of the corresponding level approximation, denoted by \(Q_\ell(\theta_\ell)\).

The uncertain coefficient field is represented in both examples by applying a pointwise transformation to an underlying Gaussian random field. Modeling a Gaussian random field is inherently an infinite-dimensional problem, since a Gaussian-distributed random variable is assigned to every point in the domain and these values are correlated through the covariance structure~\cite{chung2020introduction}. For computational purposes, we approximate this Gaussian field using a truncated Karhunen--Lo\`eve (KL) expansion with \(M\) modes, which captures the dominant covariance structure while yielding a finite-dimensional representation.

Let $\{\lambda_m\}_{m\in\mathbb{N}}$ be the eigenvalues of the covariance function $C$ and $\{b_m\}_{m\in\mathbb{N}}$ be its normalized eigenfunctions which are obtained by solving the eigenvalue problem

\begin{equation}
\int_D C(x,y)b_m(y)dy = \lambda_m b_m(x).
\end{equation}

The covariance kernel \(C\) is problem dependent. In the main elasticity
benchmark, we use a Mat\'ern covariance function, given by
\cite{dodwell_hierarchical_2015}

\begin{equation}
    C(x,y) = \dfrac{\sigma^2}{2^{\nu-1}\Gamma(\nu)}\left(\sqrt{2\nu}\dfrac{\|x-y\|_p}{\lambda}\right)^{\nu}K_{\nu}\left(\sqrt{2\nu}\dfrac{\|x-y\|_p}{\lambda}\right)
    \label{eq:matern}
    \end{equation}
    where $\Gamma(\cdot)$ is the Gamma function, $K_\nu$ the modified Bessel function of the second kind, $\sigma^2$ the variance of the field and $\lambda$ the length scale parameter. For implementation, we set the smoothness parameter to $\nu = 1.5$ and the length scale to $\lambda = 0.5$.
The supporting Darcy-flow example uses a different covariance kernel, specified
with the Darcy model in Subsection~\ref{sec:darcy_problem}.

The Karhunen--Loève expansion provides a straightforward expansion of the random, gaussian field \(G\) in terms of the eigenpairs of the covariance operator.  
In general, this expansion can be written as Equation~\eqref{eq:kl_expansion}.


\begin{equation}
G(x,\omega) = \mathbb{E}[G(\cdot,\omega)] + \sum_{m=1}^{\infty} \sqrt{\lambda_m}\xi_m(\omega)b_m(x)
\label{eq:kl_expansion}
\end{equation}

The multilevel hierarchy is defined by a sequence of discretization levels \(\ell=0,\ldots,L\). A level is specified by the spatial resolution \(R_\ell\), the KL truncation dimension \(M_\ell\), or both. Increasing \(\ell\) therefore corresponds to solving the same inverse problem with a more accurate spatial discretization and/or a richer representation of the uncertain field.

In the numerical experiments presented in this paper the number of observation points is also increased across the levels. This influences the computed likelihood and leads to the following weighted version of the likelihood~\cite{vanmechelen_multilevel_2025}
\begin{equation}
\mathcal{L}_{\ell}(u_{\ell}^{\mathrm{obs}} \mid \theta_{\ell})
\propto 
\exp\!\left[
\dfrac{-\| W_{\ell}(u^{\mathrm{obs}}) - F_{\ell}(\theta_{\ell}) \|^{2}}
{2\, N_{\ell}^{\mathrm{obs}}\, \sigma_{F}^{2}}
\right]
\label{eq:levellike}
\end{equation}
where \(N_{\ell}^{\mathrm{obs}}\) is the number of observation points used at level \(\ell\) and \(W_{\ell} : \mathbb{R}^{N_{L}^{\mathrm{obs}}} \to \mathbb{R}^{N_{\ell}^{\mathrm{obs}}}\) is a weighting operator that maps the full set of observations to the nodes of the mesh at level \(\ell\).

The PDE-specific numerical examples are introduced in Subsection~\ref{effect accuracy}, immediately before the correction-estimate and signed-Poisson diagnostics.

\section{The Multilevel Markov chain Monte Carlo algorithm}
This section introduces the multilevel Markov chain Monte Carlo (MLMCMC) algorithm proposed by \cite{dodwell_hierarchical_2015}. Subsection~\ref{ch:MLMCMC} presents the MLMCMC methodology, while Subsection~\ref{ch:subsampling rate current research} reviews the subsampling strategy originally proposed in \cite{dodwell_hierarchical_2015}.

\label{ch:MLMCMC chapt}

\subsection{The MLMCMC algorithm}
\label{ch:MLMCMC}
The MLMCMC algorithm builds on the MLMC methodology and is designed to reduce its computational cost while maintaining the accuracy of the QoI approximation. An overview of the MLMCMC algorithm is provided in Algorithm~\ref{alg:mlmcmc}. This approach constructs a hierarchy of discretization levels, denoted by $\ell = 0, \dots, L$. These levels are generated either by refining the spatial discretization of the computational grid, represented by $R_\ell$, or by increasing the truncation accuracy of the KL expansion, represented by $M_\ell$.


The first term of the telescopic sum in Equation~\eqref{MLMC} employs coarse discretizations of the forward model to obtain a computationally inexpensive but relatively rough estimate of the quantity of interest. The samples on this coarse level of the MLMCMC algorithm are generated using the Metropolis–Hastings algorithm with the preconditioned Crank–Nicolson (pCN) proposal density \cite{dodwell_hierarchical_2015}. Assuming a Gaussian prior on the unknown parameter \( \theta \), the pCN proposal takes the form
\begin{equation}
    \theta' =  \sqrt{1-\beta^2}\,\theta^n + \beta\,\psi, \quad 
    \psi \sim \mathcal{N}(0, \mathcal{C})
    \label{pcn eq}
\end{equation}
where \( \beta \in (0,1) \) is a tunable parameter controlling the step size and \( \mathcal{C} \) denotes the prior covariance operator.

The subsequent terms in the telescopic sum, denoted by \( Y \), serve as correction terms that account for the discrepancy between successive discretization levels. This yields the following telescopic sum for the MLMCMC estimator
\begin{equation}
    \begin{aligned}
        \widehat{Q}_{L} &= \frac{1}{N_{0}} \sum_{n=b_0+1}^{b_0+N_{0}} Q_0(\theta_{0}^n)
        + \sum_{\ell=1}^{L} \frac{1}{N_{\ell}} \sum_{n=b_\ell+1}^{b_\ell+N_{\ell}} 
        \bigl(Q_{\ell}(\theta_{\ell}^n) - Q_{\ell-1}(\Theta_{\ell-1}^n)\bigr) \\
        &= \widehat{Q}_0 + \sum_{\ell=1}^{L} \widehat{Y}_{\ell, N_{\ell}}
    \end{aligned}
    \label{correction term}
\end{equation}

where $\theta_\ell$ and $\Theta_{\ell-1}$ denote two correlated Markov chains on level $\ell$ and $\ell-1$ respectively.

\begin{algorithm}
\caption{Multilevel Markov chain Monte Carlo}
\begin{algorithmic}
\State Define levels $\ell=0,\ldots,L$. Sample $\theta_{L}^0 \sim \pi(\theta_L)$. For each level $\ell$, define $\theta_{\ell}^0$ as the truncation of $\theta_L^0$ after the first $M_{\ell}$ modes. Define subsampling rates $\tau_1,\ldots,\tau_L$, where \(\tau_\ell\) couples level \(\ell-1\) to level \(\ell\).
\For{$n_0=0,\ldots,N_0-1$}
\State{Generate $\theta_0^{n_0+1}$ using one iteration of the Metropolis-Hastings algorithm with}
\State{pCN proposals}
\EndFor
\For{$\ell=1,\ldots,L$}
\For{$n_{\ell} = 0,\ldots,N_{\ell}-1$}
\State Generate coarse proposal by subsampling $(\theta_{\ell}^C)' = \theta_{\ell-1}^{(n_{\ell}+1)\tau_{\ell}}$
\State Generate fine proposal $(\theta_{\ell}^F)' \sim q^{\text{pCN}}(\cdot | \theta_{\ell}^{F,n})$
\State Concatenate coarse and fine proposals $\theta_{\ell}' = \left[(\theta_{\ell}^C)',(\theta_{\ell}^F)'\right]$
\State Evaluate multilevel acceptance rate
\[
\alpha_{\ell} = \min\left\lbrace 1, \dfrac{\mathcal{L}_{\ell}(u|\theta_{\ell}')\mathcal{L}_{\ell-1}(u|\theta_{\ell}^{C,n})}{\mathcal{L}_{\ell}(u|\theta_{\ell}^n)\mathcal{L}_{\ell-1}(u|(\theta_{\ell}^C)')} \right\rbrace
\]
\State Sample $r \sim \mathcal{U}([0,1])$
\If{$\alpha_{\ell} \geq r$}
\State Accept proposal and set $\theta_{\ell}^{n+1} = \theta_{\ell}'$
\Else
\State Reject proposal and set $\theta_{\ell}^{n+1} = \theta_{\ell}^n$
\EndIf
\EndFor
\EndFor
\State \Return Markov chains of samples $\{\theta_{\ell}^{n_{\ell}}\}_{n_{\ell}=0}^{N_{\ell}}$ at all levels $\ell=0,\ldots,L$
\end{algorithmic}
\label{alg:mlmcmc}
\end{algorithm}

In contrast to standard MLMC, which relies on independently generated samples at
each level, MLMCMC constructs a Markov chain of samples on every
level~\cite{dodwell_hierarchical_2015}. Consequently, samples within a single level exhibit autocorrelation. To generate
these Markov chains, the algorithm proceeds hierarchically. At each level $\ell\geq 1$, every $\tau_{\ell}^{\text{th}}$ sample, where
\(\tau_{\ell}\) denotes the rate at which the chain at level $\ell-1$ is
subsampled, is taken from the Markov chain at the preceding, coarser level
$\ell-1$ and used as a proposal state for the Markov chain at the finer level
$\ell$.



The auxiliary chain $\Theta$ is defined as a subsampled version of the original coarse-level chain $\theta$  
\begin{equation}
\Theta_{\ell,C}^n = \theta_{\ell-1}^{n \tau_{\ell}}.
\end{equation}

However, since the parameters at different levels may not share the same dimensionality, the truncation dimension \( M_{\ell-1} \) of the parameter \( \theta \) at level \( \ell-1 \) may be smaller than the truncation \( M_\ell \) at level \( \ell \). Consequently, \( M_\ell - M_{\ell-1} \) additional components, denoted by \( \theta_{\ell,F}^n \), must be appended to \( \Theta_{\ell,C}^n \) resulting in 

\begin{equation}
    \theta^n_\ell = [\Theta_{\ell,C}^n, \theta_{\ell,F}^n].
\end{equation}

The acceptance probability for a general proposed sample in the MLMCMC framework can be expressed as a variant of the classical MH acceptance rate

    \begin{equation}
        \alpha_{\ell} = \min\left\lbrace 1, \dfrac{P_{\ell}(\theta_{\ell}'|u)q_{\ell}^{\text{ML}}(\theta_{\ell}^n|\theta_{\ell}')}{P_{\ell}(\theta_{\ell}^n|u)q_{\ell}^{\text{ML}}(\theta_{\ell}'|\theta_{\ell}^n)}\right\rbrace
        \label{eq:ml_accrate_original}
    \end{equation}

where $P_{\ell}(\theta_{\ell}|u)$ represents the $\ell^\text{th}$ level posterior density of the sample $\theta_\ell$ and $q_{\ell}^{\text{ML}}$ is the proposal density for the fine modes.

As shown in~\cite{dodwell_hierarchical_2015}, when independent proposals are used in combination with the preconditioned pCN proposal scheme, the acceptance probability simplifies to a ratio of likelihoods
\begin{equation}
    \alpha_\ell = \min \left(
        1, 
        \frac{
        \mathcal{L}_\ell(u| \theta_\ell') \mathcal{L}_{\ell-1}(u | \theta_{\ell,C}^n)
        }{
        \mathcal{L}_\ell(u| \theta_\ell^n) \mathcal{L}_{\ell-1}(u | \theta_{\ell,C}')
        }
        \right)
    \label{eq:ml_acceptance}
    \end{equation}

and for the coarsest level an acceptance rate of
\begin{equation}
     \alpha_0 = \min \left(
    1, 
    \frac{
    \mathcal{L}_0(u | \Theta'_{0})
    }{
    \mathcal{L}_0(u| \Theta_0^n)
    }
    \right).
    \end{equation}


     The closed-form expression for the likelihood, given in Equation~\eqref{ingevulde likelihood}, combined with the usage of the pCN proposal density greatly simplifies the implementation of the acceptance probability in the Metropolis-Hastings step of the MLMCMC algorithm.



\subsection{Subsampling to approximate independence}
\label{ch:subsampling rate current research}
The theoretical convergence analysis of~\cite{dodwell_hierarchical_2015} assumes that the proposal distribution for the coarse modes at fine levels, $\ell > 0$, is an independent sampler of the previous level $\ell-1$. In practice, this independence assumption is approximated by setting the subsampling rate equal to the integrated autocorrelation time (IAT), denoted \(T\). The role of \(T\) is to account for the increase in Monte Carlo variance caused by autocorrelation. For an estimator of the QoI, denoted \(\widehat{Q}\), based on \(N\) samples with sample variance \(s^2\), the correlated-sample variance is written as
\begin{equation}
    \mathbb{V}[\widehat{Q}] = \frac{s^{2} T}{N}.
    \label{mc_error_correlated}
\end{equation}
The IAT is obtained from the autocorrelation function \(\rho(t)\) by~\cite{Bia_as_2023}
\begin{equation}
T = 1 + 2 \sum_{t=1}^{\infty} \rho(t).
\label{autocorrelation_equation}
\end{equation}
In computations, \(T\) is estimated from the sampled autocorrelation function using a finite truncation of this sum. Equation~\eqref{mc_error_correlated} is often rewritten in terms of the effective sample size \(N_{\mathrm{eff}}\), defined by
\begin{equation}
    N_{\mathrm{eff}} = \frac{N}{T}.
\end{equation}
Thus a correlated chain of length \(N\) gives the same Monte Carlo error as an i.i.d.\ estimator based on only \(N_{\mathrm{eff}}\) independent samples~\cite{seiffert2024estimatingeffectivesamplesize}. The IAT-based subsampling rule uses this quantity as a proxy for decorrelation of the coarse proposals. This does not imply that the subsampled Markov chain produces truly independent samples, only that it provides an equivalent level of Monte Carlo accuracy.

\section{Finite-subsampling effects}
\label{ch:deriving efficient sub}
This section develops the rationale for choosing a subsampling rate smaller than
the IAT-based default. Subsection~\ref{justification} first provides a
theoretical explanation for why independence of the proposed samples is not the
only relevant criterion in an ideal reset-kernel setting, and identifies the
caveat that arises for the practical accepted chain.
Subsection~\ref{ch: effect subsampling} then examines the influence of the
subsampling rate on computational cost, while Subsection~\ref{effect accuracy}
treats the finite-subsampling bias as an error-budget diagnostic for the QoI
estimates.

\subsection{Reset-kernel reference calculation for finite subsampling}
\label{justification}
In the original MLMCMC construction of~\cite{dodwell_hierarchical_2015}, subsampling is used to make the coarse-level proposals approximately independent draws from the intermediate posterior \(P_{\ell-1}\). This independence assumption is what allows Equation~\eqref{eq:ml_acceptance} to be derived from the general Metropolis--Hastings ratio in Equation~\eqref{eq:ml_accrate_original}; see \cite[Lemma 3.1]{dodwell_hierarchical_2015}. For the present purpose, however, full independence is stronger than what is needed in an idealized setting. It is enough that the coarse proposal is generated by a reversible transition kernel with invariant distribution \(P_{\ell-1}\). We first record the elementary detailed-balance fact used below.

\begin{lemma}
Let $K^1$ be the transition probability of a Markov chain that satisfies detailed balance with respect to a stationary distribution $P$:
\begin{equation}
K^1(X|Y)P(Y) = K^1(Y|X)P(X).
\label{detailedbalance1}
\end{equation}
Let $K^{\tau}$ denote the $\tau$-step transition probability for any $\tau \in \mathbb{N}$. It then holds that the chain satisfies $\tau$-step detailed balance:
\begin{equation}
K^{\tau}(X|Y)P(Y) = K^{\tau}(Y|X)P(X).
\label{eq:gen_detailed_balance}
\end{equation}
\label{thm:gen_detailed_balance}
\end{lemma}

\begin{proof}
We prove the result by induction. For \(\tau=1\), it is exactly the detailed-balance condition in Equation~\eqref{detailedbalance1}. 

Let $\tau > 1$ and suppose that Equation~\eqref{eq:gen_detailed_balance} holds for $\tau-1$. Then, using the Chapman--Kolmogorov equations and the induction hypothesis,

\begin{subequations}
\begin{align}
K^{\tau}(X|Y)P(Y) &= P(Y)\int K^{\tau-1}(Z|Y)K^1(X|Z) dZ
\\
&= \int P(Z)K^{\tau-1}(Y|Z)K^1(X|Z) dZ.
\end{align}
\end{subequations}

Applying detailed balance in Equation~\eqref{detailedbalance1} to the factor \(P(Z)K^1(X|Z)\) gives

\begin{subequations}
\begin{align}
\int P(Z)K^{\tau-1}(Y|Z)K^1(X|Z) dZ &= \int P(X)K^{\tau-1}(Y|Z)K^1(Z|X) dZ
\\
&= P(X)K^{\tau}(Y|X).
\end{align}
\end{subequations}
\end{proof}


We next apply Lemma~\ref{thm:gen_detailed_balance} to an ideal reset version of the multilevel proposal. This reset kernel is an exact fine-level reference transition, not the continued-chain implementation used in the simulations. In this reference construction, the coarse component proposed at level \(\ell\) is obtained by running a reversible level-\((\ell-1)\) kernel for exactly \(\tau\) steps, starting from the coarse component of the currently accepted level-\(\ell\) state. This is not a new Metropolis--Hastings principle; it is the standard detailed-balance calculation applied to this proposal kernel~\cite{metropolis1953,hastings1970,tierney_1994}. The point is that, in the reset setting, the usual MLMCMC acceptance probability can be justified without requiring the coarse proposal to be independent.

\begin{theorem}[Exact reset reference kernel]
\label{thm:subsampled_accrate}
Let \(\ell \geq 1\), and let \(K_{\ell-1}\) be a Markov transition kernel that is reversible with respect to the coarse posterior \(P_{\ell-1}\). Suppose that, conditional on the current accepted level-\(\ell\) state
\[
\theta_\ell^n=(\theta_{\ell,C}^n,\theta_{\ell,F}^n),
\]
the reset multilevel proposal is generated by
\[
\theta_{\ell,C}' \sim K_{\ell-1}^{\tau}(\theta_{\ell,C}^n,\cdot),
\qquad
\theta_{\ell,F}' \sim q_{\ell,F}^{\mathrm{ML}}(\cdot\mid \theta_{\ell,F}^n),
\]
with the coarse and fine components proposed conditionally independently. If the proposal is accepted with the Metropolis--Hastings probability

\begin{equation}
\alpha_{\ell}
=
\min\left\lbrace
1,
\frac{
P_{\ell}(\theta_{\ell}')
P_{\ell-1}(\theta_{\ell,C}^{n})
q^{\mathrm{ML}}_{\ell,F}(\theta_{\ell,F}^{n}\mid \theta_{\ell,F}')
}{
P_{\ell}(\theta_{\ell}^{n})
P_{\ell-1}(\theta_{\ell,C}')
q^{\mathrm{ML}}_{\ell,F}(\theta_{\ell,F}'\mid \theta_{\ell,F}^{n})
}
\right\rbrace .
\label{eq:ideal_subsampled_acceptance}
\end{equation}
then the resulting reset kernel satisfies detailed balance with respect to \(P_{\ell}\), and therefore \(P_{\ell}\) is invariant. If the fine proposal is reversible with respect to the fine prior, as for pCN proposals~\cite{cotter_mcmc_2013}, then the corresponding prior proposal factors cancel in the usual way and Equation~\eqref{eq:ideal_subsampled_acceptance} reduces to the practical likelihood-ratio form used in MLMCMC.
\end{theorem}

\begin{proof}
The proposal density factorizes as
\[
q_{\ell}^{\mathrm{ML}}(\theta_\ell'\mid \theta_\ell^n)
=
K_{\ell-1}^{\tau}(\theta_{\ell,C}^n,\theta_{\ell,C}')
q_{\ell,F}^{\mathrm{ML}}(\theta_{\ell,F}'\mid \theta_{\ell,F}^n),
\]
and similarly for the reverse proposal. By Lemma~\ref{thm:gen_detailed_balance},
\[
K_{\ell-1}^{\tau}(\theta_{\ell,C}^n,\theta_{\ell,C}')
P_{\ell-1}(\theta_{\ell,C}^n)
=
K_{\ell-1}^{\tau}(\theta_{\ell,C}',\theta_{\ell,C}^n)
P_{\ell-1}(\theta_{\ell,C}').
\]
Substituting this identity into the general Metropolis--Hastings ratio~\eqref{eq:ml_accrate_original} gives~\eqref{eq:ideal_subsampled_acceptance}. The transition is therefore an ordinary Metropolis--Hastings kernel with proposal density \(q_{\ell}^{\mathrm{ML}}\), so the standard MH detailed-balance identity gives
\[
P_{\ell}(\theta_\ell^n)
q_{\ell}^{\mathrm{ML}}(\theta_\ell'\mid \theta_\ell^n)
\alpha_{\ell}(\theta_\ell^n,\theta_\ell')
=
P_{\ell}(\theta_\ell')
q_{\ell}^{\mathrm{ML}}(\theta_\ell^n\mid \theta_\ell')
\alpha_{\ell}(\theta_\ell',\theta_\ell^n).
\]
This identity proves detailed balance for the accepted moves. The rejected moves only add probability mass to the diagonal of the transition kernel, since the chain then remains at its current state. This diagonal part is balanced automatically. Therefore the full reset kernel satisfies detailed balance and leaves \(P_\ell\) invariant. If the fine-level proposal is pCN, its reversibility with respect to the Gaussian prior gives the usual cancellation of the prior and proposal-density terms, so the acceptance probability reduces to the likelihood-ratio form used above.
\end{proof}

\begin{proposition}[Scope of correctness for the practical chain]
\label{prop:scope_practical_chain}
Theorem~\ref{thm:subsampled_accrate} proves exact invariance for the reset reference kernel, not for every recursive implementation of Algorithm~\ref{alg:mlmcmc}. In the practical continued-chain implementation, a rejected fine-level proposal leaves the accepted coarse component unchanged, while the lower-level chain may continue to move. The next coarse proposal is then generated from this continued lower-level state, rather than necessarily from the coarse component of the current accepted fine state. Thus the proposal law of the practical augmented chain need not coincide with the reset proposal law in Theorem~\ref{thm:subsampled_accrate}. Unless this augmented transition is itself shown to satisfy detailed balance with an invariant law whose fine marginal is \(P_{\ell}\), the theorem should not be interpreted as an exact invariance proof for the practical finite-\(\tau\) accepted chain.
\end{proposition}

\begin{proof}
The reset kernel proposes the coarse component from \(K_{\ell-1}^{\tau}(\theta_{\ell,C}^n,\cdot)\), where \(\theta_{\ell,C}^n\) is the coarse component of the currently accepted level-\(\ell\) state. In the continued-chain implementation, the lower-level state used for the next proposal may already have moved during one or more rejected fine-level proposals. The current accepted fine state is therefore not sufficient to determine the next proposal law. One has to augment the state with at least the accepted fine state, the current lower-level state from which the next coarse proposal will be generated, and the rejection or recursive history needed to determine future proposals. On this augmented state, the actual proposal law differs from the reset proposal law whenever the continued lower-level state differs from the accepted coarse component. The detailed-balance calculation in Theorem~\ref{thm:subsampled_accrate} therefore applies directly only to the reset kernel. Exactness of the practical chain would require a separate detailed-balance or invariance proof for its actual augmented transition.
\end{proof}

\begin{remark}
The continued-chain caveat in Proposition~\ref{prop:scope_practical_chain} is not introduced by the subsampling-rate criterion proposed in this work. It is already present in the practical recursive MLMCMC implementation of~\cite{dodwell_hierarchical_2015}, where a lower-level chain continues after a rejection at the next finer level. The role of the reset calculation is therefore limited but useful: it shows that independence is not the essential requirement in the reset transition that leaves the fine-level posterior invariant. For the practical continued-chain implementation, the remaining finite-subsampling effect is treated as a perturbation and checked against the numerical error budget.
\end{remark}

Within the reset reference setting, \cite[Lemma\,3.1]{dodwell_hierarchical_2015} remains valid: Theorem~\ref{thm:subsampled_accrate} extends the same calculation to reversible coarse proposals and shows that \(\alpha_{\ell}\) does not rely on \(q^{\mathrm{ML}}_{\ell,C}\) being an independent sampler from \(P_{\ell-1}\). This differs from MLDA, where the proposal and rejection bookkeeping are changed so that detailed balance holds for the delayed-acceptance transition being implemented~\cite{lykkegaard_multilevel_2023}. Here, the reset kernel is used only as a finite-\(\tau\) reference transition for the perturbation diagnostic applied below to the practical continued-chain implementation; it is not used as a proof that the lower-level marginal in the multilevel correction is exactly \(P_{\ell-1}\).

\subsection{Effect of subsampling on the cost}
\label{ch: effect subsampling}

The impact of the subsampling rate on the computational cost of the MLMCMC algorithm can be quantified directly. Let $\mathcal{C}_\ell$ denote the cost of generating a single sample at level $\ell$. When using the subsampling approach generating one sample at level $\ell$ requires $\tau_\ell$ additional samples from level $\ell-1$, resulting in a total cost of

\begin{equation}
\mathcal{C}_{\ell}^{\text{tot}} = \mathcal{C}_{\ell} + \tau_{\ell}\mathcal{C}^{\text{tot}}_{\ell-1},
\label{eq:samplecost}
\end{equation}

where $\mathcal{C}^{\text{tot}}_0 = \mathcal{C}_0$. Recursively applying this cost formula gives

\begin{equation}
\mathcal{C}_{\ell}^{\text{tot}}
=
\mathcal{C}_{\ell}
+ \tau_{\ell}\mathcal{C}_{\ell-1}
+ \tau_{\ell}\tau_{\ell-1}\mathcal{C}_{\ell-2}
+ \cdots
+ \left(\prod_{j=1}^{\ell}\tau_j\right)\mathcal{C}_0 .
\end{equation}

As the subsampling rate increases, so too does the cost per sample. In practice, however, we are not interested in the cost per sample, but rather the cost per effective sample which depends on the IAT of the Markov chain. Denote the IAT at level~$\ell$ 
by $T_{\ell}$. The total cost per effective sample at level~$\ell$ is then

\begin{equation}
\mathcal{C}_{\ell}^{\text{eff}} = \left\lceil T_{\ell} \right\rceil \mathcal{C}_{\ell}^{\text{tot}}.
\label{eq:effective_cost}
\end{equation}

As the subsampling rate increases, $T_{\ell}$ decreases, creating two opposing effects: the cost per sample rises, but fewer samples are required to obtain one effective sample. To quantify the combined impact of these competing factors we perform a simple two-level experiment using the MLMCMC algorithm across a range of subsampling rates.

We use the elasticity beam model introduced in Subsection~\ref{sec:plane_stress_problem}, where the underlying random field is represented via a KL expansion truncated after 250 terms and synthetic observations are generated by discretising the field on an FE grid consisting of 128 $\times$ 32 elements.

To obtain these numerical results, we employ a likelihood-scaled variant of the MLMCMC algorithm (Equation~\eqref{eq:levellike}) applied on two levels (\(\ell = 0, 1\))~\cite{vanmechelen_multilevel_2025}. At the coarse level, the random field is represented using \(M_0=16\) terms in the KL expansion and discretised on a FE grid of \(16 \times 4\) elements. At the fine level we use \(M_1=40\) KL terms with a finer resolution of \(32 \times 8\) FEs. To mitigate stochastic variability, the experiment is repeated over seven independent random seeds for each value of \(\tau\). In each simulation, a coarse-level chain of length \(1.5\times 10^{7}\) samples is generated, while the fine-level chain length varies with the subsampling rate \(\tau\), which is treated as an input parameter taking values between \(3\) and \(200\).

Since we are working with a two-level MCMC algorithm, the total cost per sample at
the finest level of Equation~\eqref{eq:samplecost} simplifies to
\(\mathcal{C}_1^{\text{tot}} = \mathcal{C}_1 + \tau_1\mathcal{C}_0\). A simple
2D finite-element scaling argument suggests that, when the number of elements
grows by a factor of four between adjacent grids, the fine-level solve cost
should also grow by an approximately constant factor. This is only a simplifying
model: measured wall-clock costs can deviate because of the linear solver, KL
truncation, likelihood evaluation, and MATLAB overhead. The numerical cost
comparisons below use measured per-sample timings from the runs rather than this
nominal scaling alone. At large subsampling rates \(\tau\), efficiency gains are
therefore only obtained if \(T_1^{\tau} < T_1^1/\tau\), where \(T_1^{\tau}\) is
the IAT of the chain at level 1 when using subsampling rate \(\tau\). This decay
is not realistic, as shown by the following basic lemma which holds in the case
of autoregressive correlations of order one (AR(1)):

\begin{lemma}
Let $\{\theta^n\}_{n=0}^{\infty}$ be a Markov chain with AR(1) autocorrelation function $\rho_k = \rho^k$ with $|\rho| < 1$. Define a subsampling rate $\tau \in \mathbb{N}$ and let $T^{\tau}$ be the integrated autocorrelation time of the subsampled chain $\{\theta^{n\tau}\}_{n=0}^{\infty}$. Then it holds that
\begin{equation}
T^{\tau} \geq \dfrac{T^1}{\tau},
\label{lemma 4.3 in}
\end{equation}
with equality only when $\tau=1$.
\end{lemma}
\begin{proof}
The proof can be found in \cite{link_thinning_2012}.
\end{proof}

For more general correlation structures Inequality \eqref{lemma 4.3 in} does not always hold since the effect of subsampling on the full autocorrelation sequence can be substantially more intricate. Nevertheless, similar lower bounds can often be derived when the chain exhibits sufficiently regular decay of correlations or a well-behaved spectral density.

 In the left plot of Figure~\ref{fig:subsamp_cost} we illustrate the effect of subsampling on the IAT of the chains generated by our two-level experiment. The chain at level 0 is generated by a Metropolis-Hastings algorithm and has an IAT of approximately 976. The black line in this plot shows the IAT of the subsampled coarse chain $\{\theta_0^{n\tau}\}_{n\in\mathbb{N}}$, with the value of $T_0/\tau$ as a dashed line. For small values of \(\tau\), the two curves agree closely. 
However, as \(\tau\) increases, the IAT of the subsampled chain decreases more slowly than \(T_0/\tau\), indicating a loss of efficiency due to subsampling \(\theta_0\). Therefore, subsampling, especially at large rates $\tau$, can reduce the efficiency of the algorithm.

\begin{figure}[H]
\centering
\includegraphics[width=\textwidth]{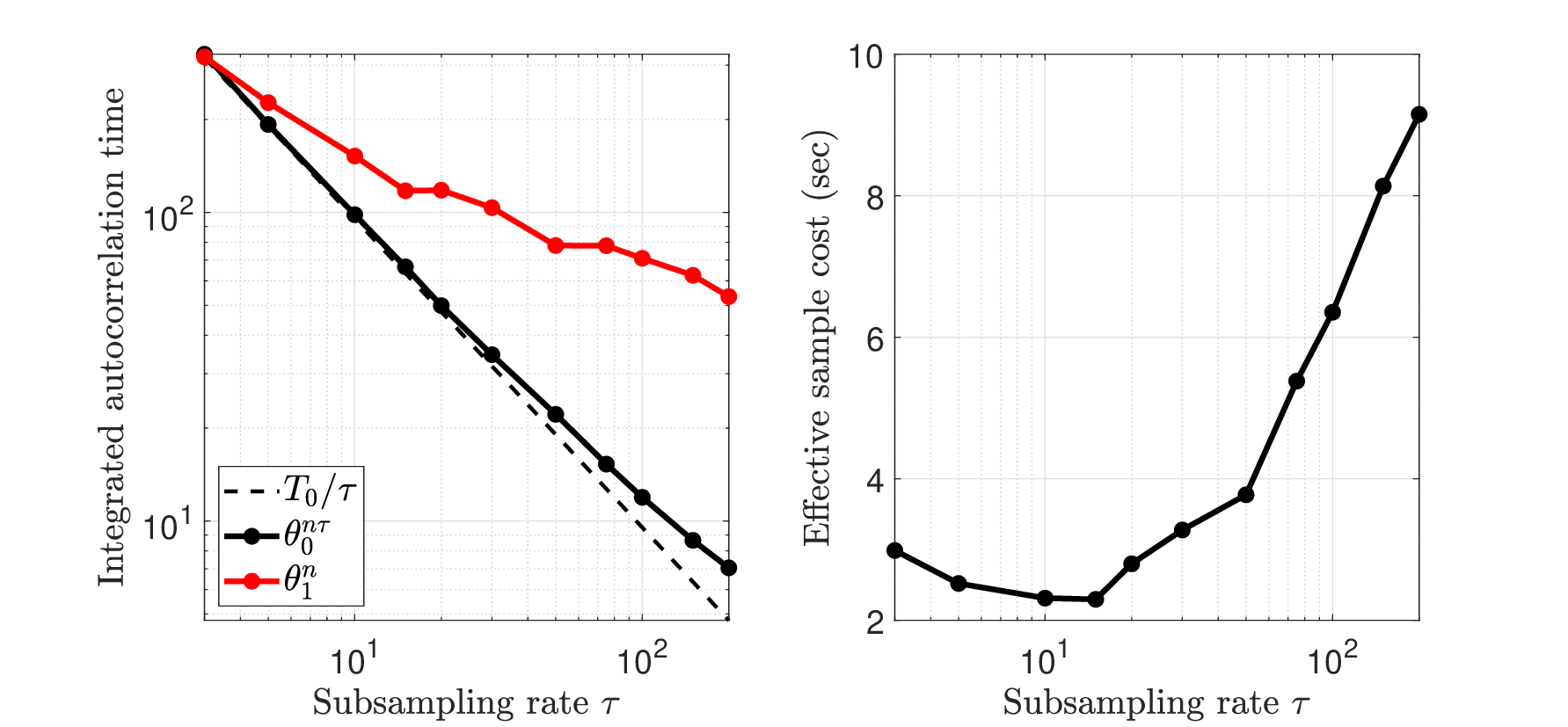}
\caption{Left: IATs of subsampled chain at level 0 (black) and full chain at level 1 (red) for varying subsampling rate. Right: cost per effective sample $\mathcal{C}_1^{\text{eff}}$ for varying subsampling rates.}
\label{fig:subsamp_cost}
\end{figure}

Note that it is not the subsampled chain $\{\theta_0^{n\tau}\}_{n\in\mathbb{N}}$ itself which is used in the multilevel algorithm. Rather, these samples are used as proposals for the coarse modes of the fine chain $\theta_1$. Because these proposals themselves are then subjected to an additional MH accept-reject step, the resulting fine level chain $\theta_1$ is more strongly correlated than the subsampled chain $\{\theta_0^{n\tau}\}_{n\in\mathbb{N}}$. The decay of the IAT at the fine level goes significantly slower than $T_0/\tau$.

The right plot of Figure~\ref{fig:subsamp_cost} shows the resulting effect of the slow decay of IAT on the cost per effective sample as calculated from Equation~\eqref{eq:effective_cost}. For low values of $\tau$ the cost initially decreases since the IAT still decreases significantly with increasing $\tau$. Additionally, the value of $\mathcal{C}_1^{\text{tot}}$ in Equation~\eqref{eq:samplecost} is still strongly influenced by $\mathcal{C}_1$ for low values of $\tau$ and the cost per sample does not scale strongly with $\tau$. As $\tau$ continues to increase however, the IAT decays more slowly while the increase in cost is more noticeable, leading to a clear increase in cost per effective sample.

 The above figure finds an optimum of $\mathcal{C}_1^{\text{eff}}$ around $\tau \approx 15$. Beyond this value the cost per effective sample increases considerably with the subsampling rate. It is important to note that the plot includes subsampling rates only up to \(\tau = 200\), which remains well below the coarse-level IAT of 976, the subsampling rate recommended in the original algorithm.
 This demonstrates that substantial computational savings can be achieved by carefully selecting $\tau$.

\subsection{Effect of subsampling on the accuracy}
\label{effect accuracy}
In addition to reducing the autocorrelation of finer levels, subsampling the coarse chain also affects the bias in the MC estimate of the fine-level correction term $Y_{\ell}$. Investigation of the bias under varying choices of subsampling rate $\tau$ has previously been carried out by \cite{dodwell_hierarchical_2015} in the original development of the MLMCMC algorithm. In their model elliptic PDE benchmark, they reported a bias on $\mathbb{E}[Y_1]$ of less than 30\% even if no subsampling was used (i.e. $\tau_1 = 1$), dropping to under 10\% even for a low subsampling rate of $\tau_1=15$ on an original IAT of $T_0\approx 86$. We compare the corresponding diagnostics for the elasticity beam benchmark and for an existing Darcy-flow benchmark. The purpose is to test whether the same autocorrelation-driven effect is visible in these benchmarks.

\subsubsection{Elasticity beam problem}
\label{sec:plane_stress_problem}
The primary numerical benchmark is based on the two-dimensional linear elasticity equations on a rectangular domain representing a cantilever. In this example, the unknown parameter $\theta$ represents the spatially varying Young's modulus $E$ of the material throughout the domain. The shear modulus is given by \(G = \frac{\theta}{2(1+\tilde{\nu})}\), where \(\tilde{\nu}\) denotes the Poisson ratio of the material. We assume \(\tilde{\nu}\) is constant and satisfies \(\tilde{\nu} < 0.5\). We introduce the Lam\'e parameters
\begin{equation}
\tilde{\lambda} = \dfrac{\theta\tilde{\nu}}{(1+\tilde{\nu})(1-2\tilde{\nu})}, \qquad \tilde{\mu} = G = \dfrac{\theta}{2(1+\tilde{\nu})}.
\label{eq:lame}
\end{equation}
Let $I$ denote the identity matrix. Combining these equations gives the linear elasticity equations,
\begin{equation}
-\nabla\cdot \left[\tilde{\lambda}(\nabla\cdot u)I + \tilde{\mu}(\nabla u + (\nabla u)^T)\right] = F,
\label{eq:linear_elasticity}
\end{equation}
where \(F: \mathbb{R}^2 \to \mathbb{R}^2\) denotes the known body forces acting on the system, while \(u: \mathbb{R}^2 \to \mathbb{R}^2\) represents the displacement field. Synthetic observations are generated on the finest elasticity discretization and are restricted to the coarser observation sets through the weighting operator used in Equation~\eqref{eq:levellike}.

\subsubsection{Darcy-flow problem}
\label{sec:darcy_problem}
The Darcy problem models flow through porous media~\cite{iqbal2015fundamentals,cliffe_multilevel_2011,dodwell_hierarchical_2015} using
\begin{equation}
q+k\nabla p=g,\qquad \nabla\cdot q=0,
\label{eq:darcy_first_order}
\end{equation}
or, equivalently,
\begin{equation}
-\nabla\cdot\left(k(x,\omega)\nabla p(x,\omega)\right)=f(x).
\label{eq:darcy_second_order}
\end{equation}
Here \(p\) is the pressure, \(q\) is the Darcy flux, and the uncertain coefficient is the permeability \(k\). For the supporting Darcy-flow experiment, \(D=[0,1]^2\), \(f\equiv -1\), and the boundary conditions are
\begin{equation}
p|_{x_1=0}=0,\qquad p|_{x_1=1}=1,\qquad
\frac{\partial p}{\partial n}\bigg|_{x_2=0}=0,\qquad
\frac{\partial p}{\partial n}\bigg|_{x_2=1}=0 .
\label{eq:darcy_bc}
\end{equation}
The permeability is lognormal and represented by a truncated KL expansion, as is standard for random fields with elliptic PDE coefficients~\cite{ghanem2003stochastic,zheng2017simulation,charrier_strong_2012}. The underlying Gaussian field uses the separable exponential covariance
\begin{equation}
C_{\mathrm{D}}(x,y)
=
\exp\left(-\frac{|x_1-y_1|}{\lambda_{\mathrm{D}}}\right)
\times
\exp\left(-\frac{|x_2-y_2|}{\lambda_{\mathrm{D}}}\right),
\qquad \lambda_{\mathrm{D}}=0.5,
\label{eq:darcy_covariance}
\end{equation}
where the product reflects the two spatial directions. The one-dimensional KL
eigenpairs are obtained from the corresponding transcendental root equation. The
resulting truncated permeability field is
\begin{equation}
k_R(x,\omega)=
\exp\left(\sum_{r=1}^{R}\sqrt{\mu_r}\,\phi_r(x)\xi_r(\omega)\right),
\label{eq:darcy_kl}
\end{equation}
and the Darcy QoI used here is the average pressure
\begin{equation}
Q_{\mathrm{D}}=\int_0^1\int_0^1 p(x_1,x_2)\,dx_1\,dx_2 .
\label{eq:darcy_qoi}
\end{equation}
To obtain the Darcy numerical results in Figures~\ref{fig:bias_effect}
and~\ref{fig:poisson_bias_diagnostic}, we use the same MLMCMC implementation
with the Darcy forward model substituted for the elasticity solver. The
comparison is performed on two levels, \(\ell=0,1\). At the coarse level, the
log-permeability field is represented using \(M_0=15\) terms in the KL
expansion and discretized on an \(8\times8\) grid. At the fine level, we use
\(M_1=30\) KL terms and a \(16\times16\) grid. Synthetic observations are
generated from a reference field with \(120\) KL terms, and the likelihood uses
\(16\) pressure observations on \([0.2,0.8]^2\) with observation noise
\(\sigma=10^{-8}\). The pCN proposal parameter is \(\beta=0.1\). For each
subsampling rate, the VSC runs use \(60\) independent random seeds. The base
chain length is set to \(\tau_1 10^5\), so the level-one correction chain has
approximately \(10^5\) samples before the \(10\%\) burn-in removal. 

The Darcy figures are used only as supporting evidence that finite subsampling
has qualitatively similar behaviour in a second elliptic inverse problem.

\subsubsection{Correction term estimate comparison}
Figure~\ref{fig:bias_effect} shows the MC estimate for $\mathbb{E}[Y_1]$ for a range of different subsampling rates in both test problems. The reference estimator in each panel is the mean correction obtained from the largest available subsampling rate in the same experimental family, using the same post-burn-in estimator for \(Y_1\). The numerical values discussed below refer to the elasticity beam benchmark. The outcome is similar to that reported in \cite{dodwell_hierarchical_2015}: while the IAT of the beam coarse chain, and thus the benchmark choice for $\tau_1$, is approximately 976, the estimates stabilize for considerably lower values of $\tau_1$. Around $\tau\approx 75$ there is little visible systematic bias remaining relative to the Monte Carlo scatter. Even at the lowest end, for $\tau_1=3$, the observed relative difference is about \(5\%\) of the correction mean. The relative effect on the final multilevel estimate for \(Q\) is smaller, since in the beam experiment the correction term \(Y_1\) is about an order of magnitude smaller than \(Q\) itself.

These data, for either benchmark, do not prove that the invariant measures of the practical finite-\(\tau\) accepted chains coincide. Part of the observed discrepancy can be finite-chain Monte Carlo error, especially at high subsampling rates where fewer fine samples are available, but a small structural perturbation of the effective transition kernel is also possible. We therefore do not claim that the low-\(\tau\) bias necessarily disappears by running longer chains unless the augmented-chain invariant measure can be shown to be unchanged. Longer runs reduce Monte Carlo uncertainty and make the perturbation easier to estimate; they do not, by themselves, prove absence of perturbation bias.

\begin{figure}[H]
\centering
\begin{minipage}[t]{0.48\textwidth}
\centering
\includegraphics[width=\textwidth]{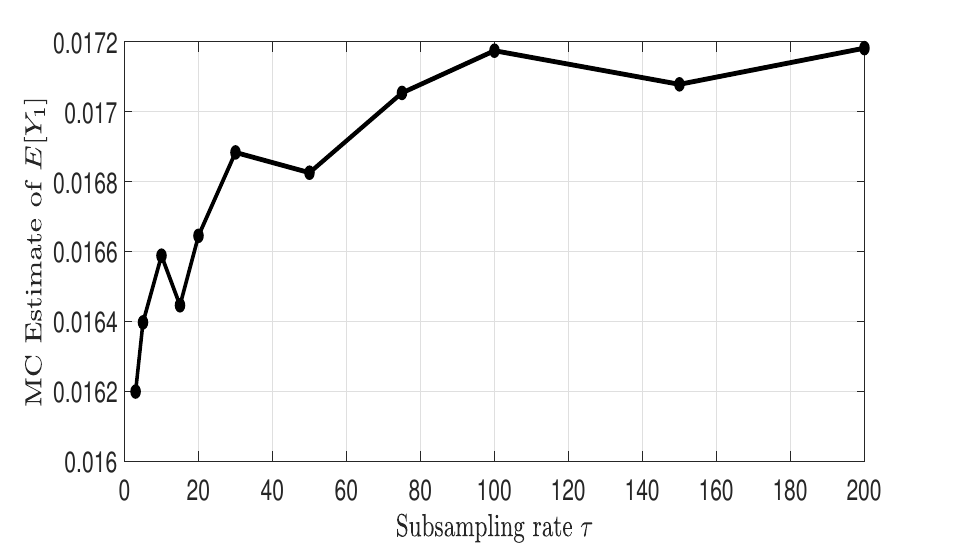}\\
Elasticity beam
\end{minipage}\hfill
\begin{minipage}[t]{0.48\textwidth}
\centering
\includegraphics[width=\textwidth]{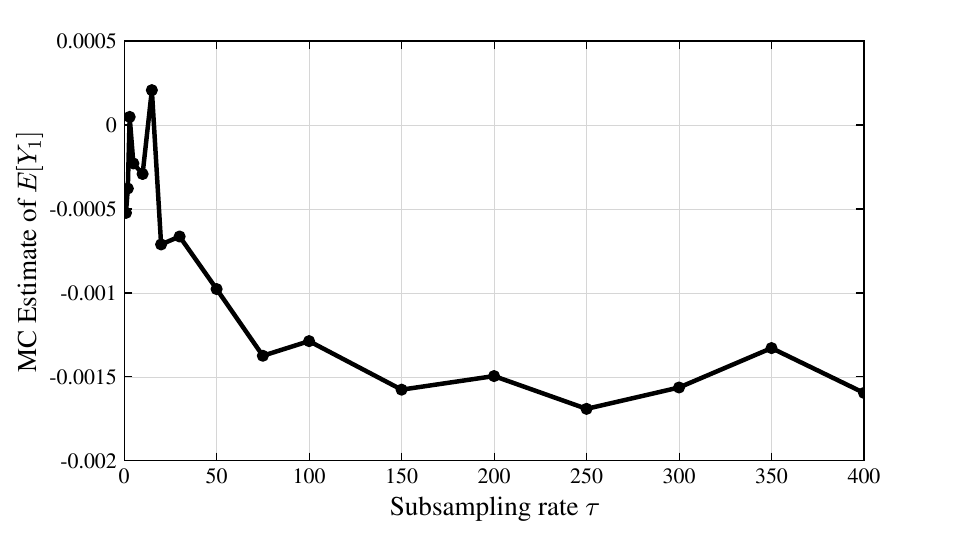}\\
Darcy flow
\end{minipage}
\caption{MC estimate of the correction term $Y_1$ in the multilevel hierarchy for varying subsampling rate $\tau_1$. Left: elasticity beam benchmark. Right: Darcy-flow benchmark.}
\label{fig:bias_effect}
\end{figure}

In what follows, we compare the finite-subsampling effect with the dominant
numerical error scales. For the rate selected by the proposed tuning rule, the
perturbation indicators are below the finest-level discretization proxy used in
the reported multilevel error budget, but they are not below the paired Monte
Carlo error of the level-one correction alone. Here, the finest-level
discretization proxy is the absolute estimated mean of the last multilevel
correction, \(\widehat d_L=|\widehat{\mathbb{E}}[Y_L]|\). It measures the change
in the posterior QoI expectation between the two finest discretization levels.
When the mean corrections decay regularly under refinement, this quantity
provides an empirical indication of the remaining discretization bias, rather
than a direct estimate or a rigorous bound. The diagnostic therefore shows a
visible finite-subsampling effect that remains smaller than the reported
discretization scale, rather than an effect hidden by Monte Carlo noise.

\subsection{Perturbation bias and a signed Poisson diagnostic}
\label{sec:poisson_bias_diagnostic}
Subsection~\ref{justification} shows that the reset reference kernel is exact,
while the practical continued-chain implementation can perturb the transition
seen by the correction chain. We quantify this finite-subsampling effect on the
augmented state \(X_n^\tau\), which contains the accepted level-\(\ell\) state,
the live lower-level state used for future coarse proposals, and the bookkeeping
needed to determine the next proposal. Let \(f(X_n^\tau)=Y_\ell\). If the
augmented process has transition kernel \(\mathcal{K}_{\tau}\) and invariant
distribution \(\mu_{\tau}\), its bias relative to a reference augmented kernel
\(\mathcal{K}_{\star}\), with invariant distribution \(\mu_{\star}\), is
\(B_{\tau}=\mu_{\tau}(f)-\mu_{\star}(f)\).

The Poisson-equation identity below is a standard way to express this
bias~\cite{glynn_meyn_1996,mijatovic_vogrinc_2018}; the total-variation bound
uses the usual dual characterization of total variation~\cite[Section~4.1]{levin_peres_wilmer_2017}. The argument is also in the spirit of
perturbation bounds for uniformly ergodic chains~\cite{mitrophanov_2005} and
approximate kernels in noisy MCMC~\cite{alquier_noisy_2016}. The argument is
separated into a general perturbation identity, an algorithm-specific kernel
bound, and the resulting bias bound. The kernel bound uses the following
state-averaged forgetting condition.

\begin{hypothesis}[State-averaged lower-level forgetting]
\label{hyp:lower_level_forgetting}
For a fixed level \(\ell\), let \(K_{\ell-1}\) denote the level-\((\ell-1)\)
kernel used to generate coarse proposals. Let
\(z=(\theta_{\ell},\eta_{\ell-1},h)\) denote a generic state of the practical
augmented chain in stationarity, with law \(\mu_{\tau}\). Here
\(\eta_{\ell-1}\) is the live lower-level state used for the next coarse
proposal, \(h\) collects the remaining bookkeeping variables, and
\(\theta_{\ell,C}\) is the coarse component of the accepted level-\(\ell\)
state. Define
\begin{equation}
    \Delta_{\ell}(\tau)
    =
    \int
    \left\|
    K_{\ell-1}^{\tau}(\eta_{\ell-1},\cdot)
    -
    K_{\ell-1}^{\tau}(\theta_{\ell,C},\cdot)
    \right\|_{\mathrm{TV}}
    \,\mu_{\tau}(dz).
\end{equation}
There exist constants \(A_{\ell}<\infty\) and \(0<\rho_{\ell}<1\) such that,
for every \(\tau\geq1\),
\begin{equation}
    \Delta_{\ell}(\tau) \leq A_{\ell}\rho_{\ell}^{\tau}.
\label{eq:lower_level_forgetting}
\end{equation}
\end{hypothesis}

This assumption concerns forgetting between the two lower-level states
encountered in the practical and reset constructions, rather than uniform
mixing over the complete state space. We do not verify the condition
theoretically for the numerical benchmarks; the signed-Poisson calculation below
is therefore used as an empirical error-budget diagnostic.

\begin{theorem}[Poisson perturbation identity on the augmented chain]
\label{thm:poisson_perturbation}
Let \(\mathcal{K}_{\tau}\) be the transition kernel of the practical augmented
chain and let \(\mathcal{K}_{\star}\) be a reference transition kernel on the
same augmented state space. Assume that \(\mathcal{K}_{\star}\) is ergodic and
that the centered correction
\[
\bar f=f-\mu_{\star}(f)
\]
admits a bounded solution \(g_{\star}\) of the Poisson equation
\begin{equation}
g_{\star}-\mathcal{K}_{\star}g_{\star}=\bar f .
\label{eq:poisson_equation_bias}
\end{equation}
If \(\mu_{\tau}\) is invariant for \(\mathcal{K}_{\tau}\), then the bias in the
expected correction can be written, by the standard Poisson-equation
perturbation identity~\cite{glynn_meyn_1996}, as
\begin{equation}
B_{\tau}
=
\mu_{\tau}\bigl((\mathcal{K}_{\tau}-\mathcal{K}_{\star})g_{\star}\bigr).
\label{eq:signed_poisson_identity}
\end{equation}
In particular, it satisfies the total-variation bound
\begin{equation}
\left|B_{\tau}\right|
\leq
\mathrm{osc}(g_{\star})
\int
\left\|
\mathcal{K}_{\tau}(x,\cdot)
-
\mathcal{K}_{\star}(x,\cdot)
\right\|_{\mathrm{TV}}
\,\mu_{\tau}(dx),
\label{eq:poisson_tv_bound}
\end{equation}
where \(\mathrm{osc}(g_{\star})=\sup_x g_{\star}(x)-\inf_x g_{\star}(x)\).
\end{theorem}

\begin{proof}
Since \(g_{\star}\) solves~\eqref{eq:poisson_equation_bias},
\[
B_{\tau}
=
\mu_{\tau}(\bar f)
=
\mu_{\tau}\bigl(g_{\star}-\mathcal{K}_{\star}g_{\star}\bigr).
\]
Using the invariance relation
\(\mu_{\tau}\mathcal{K}_{\tau}=\mu_{\tau}\), we also have
\(\mu_{\tau}(g_{\star})=\mu_{\tau}(\mathcal{K}_{\tau}g_{\star})\). Therefore
\[
B_{\tau}
=
\mu_{\tau}\bigl((\mathcal{K}_{\tau}-\mathcal{K}_{\star})g_{\star}\bigr),
\]
which proves~\eqref{eq:signed_poisson_identity}. The total-variation bound
follows by applying the definition of total variation to the signed measure
\(\mathcal{K}_{\tau}(x,\cdot)-\mathcal{K}_{\star}(x,\cdot)\), integrating the
result with respect to \(\mu_{\tau}\), and using the oscillation seminorm of
\(g_{\star}\).
\end{proof}

\begin{proposition}[Averaged augmented-kernel closeness]
\label{prop:mixing_to_kernel_perturbation}
Fix level \(\ell\) and a subsampling rate \(\tau\). Let
\(z=(\theta_{\ell},\eta_{\ell-1},h)\) denote an augmented practical state, where
\(\eta_{\ell-1}\) is the live lower-level state used to generate the next coarse
proposal and \(h\) denotes the remaining bookkeeping variables. Let
\(\theta_{\ell,C}\) be the coarse component of the accepted level-\(\ell\) state.
Let \(\mathcal{K}_{\ell,\tau}^{\mathrm{cont}}\) be the practical
continued-chain augmented kernel. Let
\(\mathcal{K}_{\ell,\tau}^{\mathrm{reset}}\) be the corresponding reference
kernel that uses the same fine proposal, acceptance probability, and
accept--reject bookkeeping, but starts the \(\tau\)-step lower-level proposal
segment from \(\theta_{\ell,C}\) instead of from \(\eta_{\ell-1}\). If
Hypothesis~\ref{hyp:lower_level_forgetting} holds, then
\begin{equation}
    \int
    \left\|
    \mathcal{K}_{\ell,\tau}^{\mathrm{cont}}(z,\cdot)
    -
    \mathcal{K}_{\ell,\tau}^{\mathrm{reset}}(z,\cdot)
    \right\|_{\mathrm{TV}}
    \,\mu_{\tau}(dz)
    \leq A_{\ell}\rho_{\ell}^{\tau}.
\label{eq:mixing_to_kernel_bound}
\end{equation}
\end{proposition}

\begin{proof}
For fixed \(z\), let \(Q_{\tau,z}^{\mathrm{cont}}\) and
\(Q_{\tau,z}^{\mathrm{reset}}\) be the proposal laws before the common
accept--reject step. The only different factor is the coarse proposal marginal,
which is \(K_{\ell-1}^{\tau}(\eta_{\ell-1},\cdot)\) in the continued chain and
\(K_{\ell-1}^{\tau}(\theta_{\ell,C},\cdot)\) in the reset chain. The remaining
proposal variables, including the fine-level proposal and the bookkeeping
variables needed for the subsequent accept--reject step, are generated in the
same way in both constructions. Thus, if \(M\) denotes this common conditional
kernel for the remaining proposal variables given the coarse proposal, the two
proposal laws can be written as
\[
Q_{\tau,z}^{\mathrm{cont}}
=
K_{\ell-1}^{\tau}(\eta_{\ell-1},\cdot)M,
\qquad
Q_{\tau,z}^{\mathrm{reset}}
=
K_{\ell-1}^{\tau}(\theta_{\ell,C},\cdot)M .
\]
Total variation distance is contracted by a common Markov
kernel~\cite[Chapter~4]{levin_peres_wilmer_2017}: for any two probability
measures \(\nu\) and \(\nu'\),
\[
\|\nu M-\nu' M\|_{\mathrm{TV}}
\leq
\|\nu-\nu'\|_{\mathrm{TV}} .
\]
Applying this property with
\(\nu=K_{\ell-1}^{\tau}(\eta_{\ell-1},\cdot)\) and
\(\nu'=K_{\ell-1}^{\tau}(\theta_{\ell,C},\cdot)\) gives
\[
\left\|
Q_{\tau,z}^{\mathrm{cont}}
-
Q_{\tau,z}^{\mathrm{reset}}
\right\|_{\mathrm{TV}}
\leq
\left\|
K_{\ell-1}^{\tau}(\eta_{\ell-1},\cdot)
-
K_{\ell-1}^{\tau}(\theta_{\ell,C},\cdot)
\right\|_{\mathrm{TV}}.
\]
The full transition kernels \(\mathcal{K}_{\ell,\tau}^{\mathrm{cont}}\) and
\(\mathcal{K}_{\ell,\tau}^{\mathrm{reset}}\) are obtained from these proposal
laws by applying the accept--reject and bookkeeping update.
The common MLMCMC acceptance probability and bookkeeping then define another
common Markov map from proposal variables to the next augmented state. Applying
the same contraction property once more, the same bound holds for
\(\mathcal{K}_{\ell,\tau}^{\mathrm{cont}}(z,\cdot)\) and
\(\mathcal{K}_{\ell,\tau}^{\mathrm{reset}}(z,\cdot)\). Integrating over
\(\mu_{\tau}(dz)\) and applying Hypothesis~\ref{hyp:lower_level_forgetting}
proves~\eqref{eq:mixing_to_kernel_bound}.
\end{proof}

\begin{corollary}[Geometric perturbation bias under lower-level forgetting]
\label{cor:geometric_poisson_bias}
Assume the hypotheses of Theorem~\ref{thm:poisson_perturbation} with
\(\mathcal{K}_{\tau}=\mathcal{K}_{\ell,\tau}^{\mathrm{cont}}\) and
\(\mathcal{K}_{\star}=\mathcal{K}_{\ell,\tau}^{\mathrm{reset}}\). If
Hypothesis~\ref{hyp:lower_level_forgetting} holds, then
\begin{equation}
\left|B_{\tau}\right|
\leq
A_{\ell}\,\mathrm{osc}(g_{\star})\,\rho_{\ell}^{\tau}.
\label{eq:geometric_poisson_bias}
\end{equation}
\end{corollary}

\begin{proposition}[Bound-based subsampling rate]
\label{prop:bias_admissible_rate}
Let \(\varepsilon_{\ell}^{\mathrm{bias}}>0\) be the part of the level-\(\ell\)
error budget reserved for the finite-subsampling perturbation. Suppose that a
computable quantity \(b_{\ell}(\tau)\) satisfies
\[
    |B_{\tau}| \leq b_{\ell}(\tau)
\]
for the correction observable \(Y_{\ell}\). If the subsampling rate satisfies
\begin{equation}
    b_{\ell}(\tau_{\ell}) \leq \varepsilon_{\ell}^{\mathrm{bias}},
\label{eq:bias_admissibility_condition}
\end{equation}
then the finite-subsampling perturbation of the expected level correction is
controlled within the prescribed error budget. In particular, under
Corollary~\ref{cor:geometric_poisson_bias}, it is sufficient to choose
\begin{equation}
    \tau_{\ell}
    \geq
    \max\left\{
    1,
    \left\lceil
    \frac{
    \log\left(A_{\ell}\mathrm{osc}(g_{\star})/
    \varepsilon_{\ell}^{\mathrm{bias}}\right)
    }{-\log\rho_{\ell}}
    \right\rceil
    \right\}.
\label{eq:analytic_bias_rate}
\end{equation}
\end{proposition}

\begin{proof}
The first statement follows directly from the assumed bound on \(B_{\tau}\). The
second statement is obtained by requiring the right-hand side of
Equation~\eqref{eq:geometric_poisson_bias} to be no larger than
\(\varepsilon_{\ell}^{\mathrm{bias}}\) and solving for \(\tau_{\ell}\).
\end{proof}

The integrated total-variation bound~\eqref{eq:poisson_tv_bound} may still be
conservative because it ignores signed cancellations. For the
numerical study we therefore also compute a projected signed diagnostic based
on~\eqref{eq:signed_poisson_identity}. Let \(s\) denote a correction-state bin
and \(c\) a coarse-state bin. From a reference run we estimate the binned
transition matrix \(\widehat{\mathcal{K}}_{\star}\), its invariant distribution
\(\widehat{\mu}_{\star}\), and the binned reward \(f\), and solve
\[
\widehat g_{\star}
-
\widehat{\mathcal{K}}_{\star}\widehat g_{\star}
=
f-\widehat{\mu}_{\star}(f).
\]
For each subsampling rate \(\tau\), we estimate the conditional next-step
kernels for the binned pair \((s,c)\), denoted by
\(\widehat{\mathcal{K}}_{\tau}(s,c,\cdot)\) and
\(\widehat{\mathcal{K}}_{\star}(s,c,\cdot)\), and the empirical pair
distribution \(\widehat{\eta}_{\tau}(s,c)\). The plug-in signed Poisson estimate
is
\begin{equation}
\widehat L_{\tau}
    =
    \sum_{s,c}
    \widehat{\eta}_{\tau}(s,c)
    \left(
        \widehat{\mathcal{K}}_{\tau}(s,c,\cdot)
        -
        \widehat{\mathcal{K}}_{\star}(s,c,\cdot)
    \right)
    \widehat g_{\star}.
\label{eq:signed_poisson_estimator}
\end{equation}
Placing absolute values around the summands in
\eqref{eq:signed_poisson_estimator} gives a more conservative projected
perturbation measure. We report the signed form because it estimates the
first-order contribution to the reference-relative correction bias. In the sense of
Proposition~\ref{prop:bias_admissible_rate}, let \(\widehat b_{\ell}(\tau)\)
denote either \(|\widehat L_{\tau}|\), a conservative absolute version of
\eqref{eq:signed_poisson_estimator}, or a fitted model \(c\rho^\tau\). The
empirical diagnostic rate is then
\begin{equation}
\tau_{\ell}^{\mathrm{diag}}
=
\min\left\{m\in\mathbb{N}:
\widehat b_{\ell}(m) \leq \varepsilon_{\ell}^{\mathrm{bias}}\right\}.
\label{eq:empirical_bias_rate}
\end{equation}
This rate is used to screen the cost-based choice for finite-subsampling
effects that appear larger than the intended numerical error scale in the
chosen projection.

\begin{figure}[H]
\centering
\begin{minipage}[t]{0.48\textwidth}
\centering
\includegraphics[width=\textwidth]{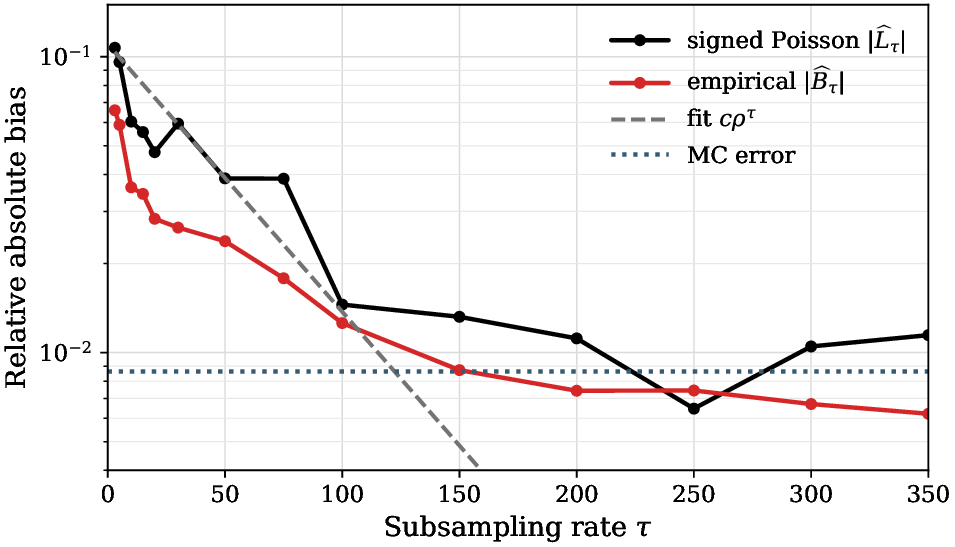}\\
Elasticity beam
\end{minipage}\hfill
\begin{minipage}[t]{0.48\textwidth}
\centering
\includegraphics[width=\textwidth]{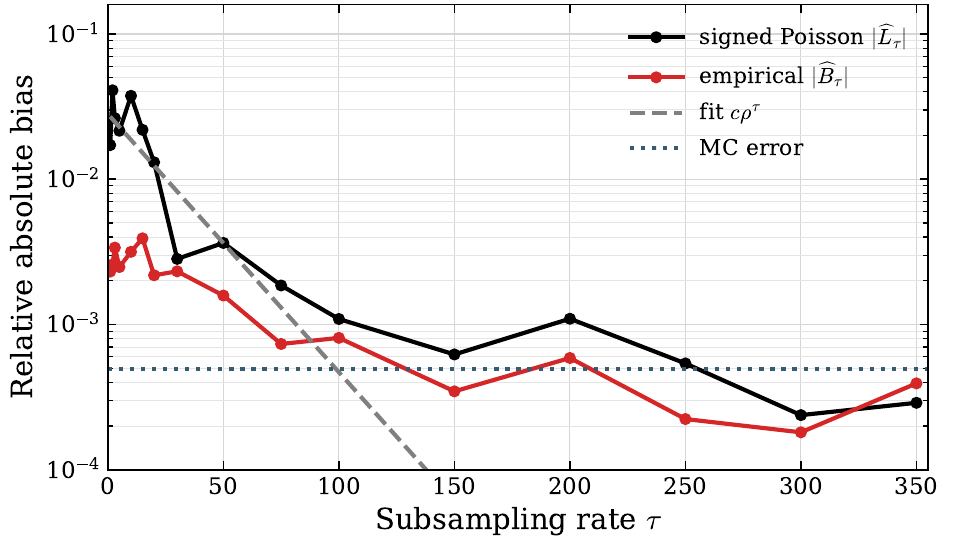}\\
Darcy flow
\end{minipage}
\caption{Empirical signed-Poisson diagnostic for the level-one correction \(Y_1\). In both panels, the solid black curve shows the projected signed quantity \(|\widehat L_{\tau}|\), the solid coloured curve shows the empirical absolute difference relative to the largest available finite-\(\tau\) reference run, and the horizontal line is the Monte Carlo error estimated from the same subset of seeds. Left: elasticity beam benchmark, with exponential fit \(c=1.90\times10^{-3}\), \(\rho=0.979\). Right: Darcy-flow benchmark, with exponential fit \(c=2.81\times10^{-2}\), \(\rho=0.960\).}
\label{fig:poisson_bias_diagnostic}
\end{figure}

The left panel of Figure~\ref{fig:poisson_bias_diagnostic} shows the diagnostic
for the level-one correction in the linear-elasticity test case. The reference
run uses the largest available subsampling rate, and the curve is computed from
the same subset of seeds used in the empirical bias study. Over the relevant
range, the signed Poisson estimate is well described by
\begin{equation}
|\widehat L_{\tau}| \approx c\rho^{\tau}.
\label{eq:poisson_exponential_fit}
\end{equation}
Fitting~\eqref{eq:poisson_exponential_fit} using the diagnostic values at
\(\tau=3,5,50\) gives \(c=1.90\times10^{-3}\) and \(\rho=0.979\),
which gives an inexpensive extrapolation of the finite-subsampling effect. The
right panel gives the corresponding diagnostic for the supporting Darcy-flow
experiment.

For the rates selected by the adaptive procedure in the beam problem,
\(\tau\approx 9\)--\(10\), the diagnostic is evaluated as an empirical
error-budget check. At \(\tau=10\), the signed Poisson diagnostic is
\(1.10\times10^{-3}\) and the paired empirical signed bias is
\(7.21\times10^{-4}\). Both are below the finest-level discretization proxy
\(1.85\times10^{-3}\), computed as the absolute mean of the \(Y_3\) correction
estimates from the four-level beam reference runs at \(\tau=400\), averaged
over the common seeds after discarding burn-in, but they are not below the paired Monte Carlo error
\(1.07\times10^{-4}\) for the level-one correction alone. Thus the
finite-subsampling effect is visible in the beam experiment, but remains below
the dominant discretization scale used in the reported multilevel error budget.
The Darcy-flow diagnostic gives the same qualitative message: the reduced
subsampling rates introduce a detectable finite-subsampling effect, but its
size remains small compared with the numerical error scale used for the
supporting experiment.

The signed Poisson diagnostic should be interpreted as an empirical screening
tool rather than as a general a posteriori certification of finite-subsampling
bias. Its value depends on practical choices made in the projected calculation,
including the selected correction-state and coarse-state bins, the reference run
used to estimate \(\widehat K_{\star}\) and \(\widehat\mu_{\star}\), and the
available Monte Carlo sample size in each bin. These choices do not invalidate
the diagnostic, but they mean that it measures the perturbation visible through
the chosen finite-dimensional projection, not the full augmented-chain bias in a
model-independent sense. We therefore use it only as a transparent empirical
error-budget check for the reported experiments.

\section{An efficient subsampling rate}
\label{sec:subsamp_derivation}
Once independence is no longer required, the subsampling rate becomes a design
parameter that governs the trade-off between autocorrelation and computational
cost across levels. The objective is thus not to approximate independence, but
to balance multilevel error contributions in a computationally efficient manner.

Accordingly, this section focuses on identifying a subsampling rate that
approximately improves the overall efficiency of the MLMCMC algorithm under the
estimated cost and autocorrelation model. Subsection~\ref{deriving} derives and
justifies an expression for such an efficient subsampling rate, while
Subsection~\ref{adaptive alg} presents an adaptive implementation of the
proposed strategy.

\subsection{Deriving an efficient rate}
\label{deriving}
The sampling error of the multilevel MCMC algorithm is given by the sum of the sampling errors of the quantity of interest and its corrections at all levels, i.e.

\begin{equation}
\varepsilon^2 = \sum_{\ell=0}^L \varepsilon^2_{\ell} = \sum_{\ell=0}^L \dfrac{s_{\ell}^2}{N_{\ell}^{\text{eff}}}
\label{eq:ml_samperror}
\end{equation}

where $s_{\ell}^2$ is the sample variance of the estimator of $Y_{\ell}$, with $Y_0=Q_{0}$.

We will first look at the two-level case. As the subsampling rate $\tau$ increases more computation time is spent on the coarse level and less on the fine level. As a result, for a fixed computational budget the coarse term $s_0^2/N_0^{\text{eff}}$ will decrease due to the increased effective sample size. The trade-off is that the fine term $s_1^2/N_1^{\text{eff}}$ will be based on fewer samples. While these samples will be more decorrelated, as shown in the previous section, this decorrelation does not sufficiently compensate for the smaller sample size, leading to an increase in the sampling error $s_1^2/N_1^{\text{eff}}$.

We illustrate this trade-off in the left plot of Figure~\ref{fig:error_terms} for the same two-level experiment as introduced in the previous section. This time, we run the MLMCMC algorithm for a fixed computational budget of $10^4$ seconds of wall-clock time per seed run, again using 7 independent seed runs executed in parallel. The figure shows the contributions of the level sampling errors $\varepsilon_0$ and $\varepsilon_1$.

All numerical experiments were performed on the batch partition of the
KU Leuven/UHasselt Tier-2 wICE cluster, using nodes equipped with Intel Xeon
Platinum 8360Y processors at 2.4 GHz and 256 GiB RAM per node~\cite{wicehardware}.
The same hardware configuration was used throughout the paper.

\begin{figure}
\centering
\includegraphics[width=\textwidth]{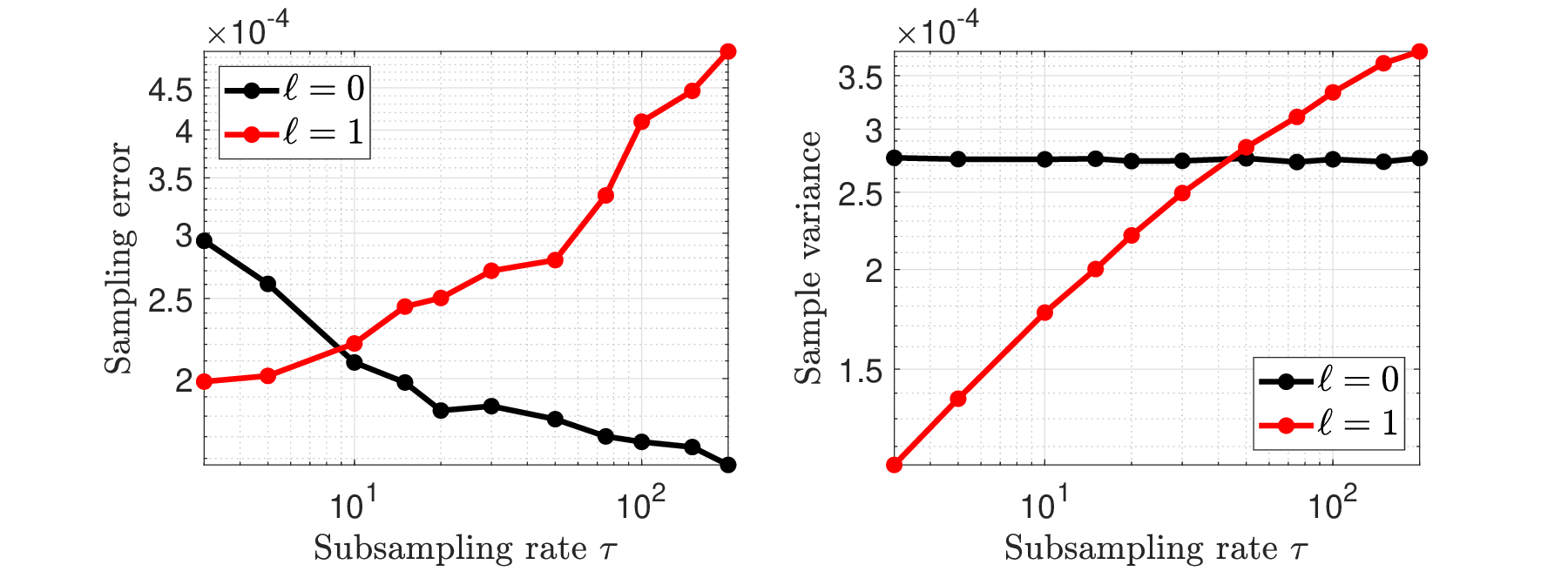}
\caption{Left: sampling errors $\varepsilon_{\ell}$ of coarse and fine estimates $\widehat{Y}^{\text{MC}}_{\ell}$ after $10^4$ seconds of wall-clock time per seed run for a varying subsampling rate $\tau$. Right: sample variance for varying subsampling rates.}
\label{fig:error_terms}
\end{figure}

From Figure~\ref{fig:error_terms}, two distinct regimes can be identified which are separated around \(\tau \approx 10\) for this particular experiment.  
At small subsampling rates, the total sampling error is dominated by the coarse-level contribution, with \(\varepsilon_0 > \varepsilon_1\). This occurs because a larger fraction of the computational budget is devoted to the fine level when \(\tau\) is small.  
In contrast, for larger values of \(\tau\), the situation reverses: the fine-level error dominates, with \(\varepsilon_1 > \varepsilon_0\). In this regime, the high subsampling rate reduces the number of fine-level samples more rapidly than the corresponding gain in decorrelation, leading to an overall increase in the fine-level sampling error.

Additionally, we see that as $\tau$ increases the fine-level sample variance $s_1^2$ increases as well. This is shown in the right plot of Figure~\ref{fig:error_terms} and can be intuitively explained as follows: when successive rejections occur at the fine level, the coarse and fine chains decouple and are only recoupled upon the next acceptance \cite{dodwell_hierarchical_2015}. The correction term is then based on less strongly correlated samples and may increase. Naturally, for large $\tau$ the coarse chain evolves for a longer time until the next acceptance and the decoupling is stronger. The coarse sample variance $s_0^2$ is independent of $\tau$, thus this effect only influences $\varepsilon_1$.

In both regimes, it can be argued that the MLMCMC method performs a certain degree of redundant calculations when trying to reach a prescribed threshold for the sampling error. This redundancy is sketched in Figure~\ref{fig:subsamp_sketch}. In the regime of low $\tau$, at a certain point the surplus error is entirely explained by the larger value of $\varepsilon_0$ over $\varepsilon_1$. At this point it would thus be more efficient to continue evolving only the coarse chain and stop the subsampling procedure to  generate fine samples. This approach decreases $\varepsilon$ for less computational cost than simply continuing the full algorithm which spends part of its budget in decreasing the already sufficiently small $\varepsilon_1$.

In the high $\tau$ regime, surplus sampling error is explained by a large $\varepsilon_1$. If we wish to decrease this, we thus require more fine samples. However, due to the high subsampling rate these samples are slow to generate and the method spends a great deal of time on the coarse level, generating long subchains to use as proposal moves. These many coarse samples provide little extra benefit, as they can only be used to decrease $\varepsilon_0$, which is already sufficiently small.

Our proposal for the choice of subsampling rate $\tau$ is then the value which
aims to reduce redundancy by making the chains reach their target thresholds
for the sampling error at approximately the same time. To determine this value
of $\tau$, we use a classical result from MLMC sampling
\cite{cliffe_multilevel_2011,giles_multilevel_2008}: we first assume that we
can sample independently from the distributions of \(Y_{\ell}\) for all
\(\ell=0,\ldots,L\) with a cost per sample of \(\mathcal{C}_{\ell}\) at all
levels. Under this allocation model, the cost to achieve a total sampling
error, as defined in Equation~\eqref{eq:ml_samperror}, is minimised if the
sample sizes \(N_{\ell}\) at each level satisfy

\begin{equation}
N_{\ell} = \dfrac{2}{\varepsilon^2}\left(\sum_{k=0}^L \sqrt{s_{k}^2\mathcal{C}_{k}}\right)\sqrt{\dfrac{s_{\ell}^2}{\mathcal{C}_{\ell}}}.
\label{eq:optimal_n}
\end{equation}

In the case of correlated samples obtained from e.g. MCMC sampling, this
expression no longer holds for the sample size \(N_{\ell}\) itself, but rather
the effective sample size \(N_{\ell}^{\mathrm{eff}}\). Similarly, the cost per
sample \(\mathcal{C}_{\ell}\) should be replaced by the cost per effective
sample \(\mathcal{C}_{\ell}^{\mathrm{eff}}\), yielding

\begin{equation}
N_{\ell}^{\text{eff}} = \dfrac{2}{\varepsilon^2}\left(\sum_{k=0}^L \sqrt{s_{k}^2\mathcal{C}_{k}^{\text{eff}}}\right)\sqrt{\dfrac{s_{\ell}^2}{\mathcal{C}_{\ell}^{\text{eff}}}}.
\end{equation}

\begin{figure}
\centering
\begin{tikzpicture}[scale=0.7]
\draw[ultra thick] (0,6) -- (6,6);
\draw[ultra thick, gray] (6,6) -- (11,6);
\draw[ultra thick] (0,7) -- (11,7);
\draw[ultra thick] (0,3) -- (11,3);
\draw[ultra thick] (0,1) -- (9,1);
\draw[ultra thick] (0,0) -- (9,0);
\draw[ultra thick] (0,4) -- (0.5,4);
\draw[ultra thick] (1.5,4) -- (2.5,4);
\draw[ultra thick] (3.5,4) -- (4.5,4);
\draw[ultra thick] (5.5,4) -- (6,4);
\draw[ultra thick, gray] (6,4) -- (6.5,4);
\draw[ultra thick, gray] (7.5,4) -- (8.5,4);
\draw[ultra thick, gray] (9.5,4) -- (10.5,4);
\draw[ultra thick, dashed] (0.5,4) -- (1.5,4);
\draw[ultra thick, dashed] (2.5,4) -- (3.5,4);
\draw[ultra thick, dashed] (4.5,4) -- (5.5,4);
\draw[ultra thick, dashed, gray] (6.5,4) -- (7.5,4);
\draw[ultra thick, dashed, gray] (8.5,4) -- (9.5,4);
\draw[ultra thick, dashed, gray] (10.5,4) -- (11,4);
\foreach \x in {0,...,6}
{
\filldraw[black] (\x,6) circle (4pt);
\draw[ultra thick] (\x,6) -- (\x,7);
}
\foreach \x in {7,...,11}
{
\filldraw[gray] (\x,6) circle (4pt);
\draw[ultra thick, gray] (\x,6) -- (\x,7);
}
\foreach \x in {0,4,8,12,16,20}{
\draw[ultra thick] (\x/2,4) -- (\x/2,3);
\filldraw[black] (\x/2,3) circle (4pt);
}
\foreach \x in {0,...,6}
{
\filldraw[black] (1.5*\x,0) circle (4pt);
\draw[ultra thick] (1.5*\x,0) -- (1.5*\x,1);
}
\foreach \x in {0,...,12}
{
\filldraw[black] (\x/2,7) circle (4pt);
\filldraw[black] (\x/2,1) circle (4pt);
}
\foreach \x in {13,...,18}
{
\filldraw[black] (\x/2,7) circle (4pt);
\filldraw[black] (\x/2,1) circle (4pt);
}
\foreach \x in {19,...,22}
{
\filldraw[black] (\x/2,7) circle (4pt);
}
\foreach \x in {0,1,3,4,5,7,8,9,11,12}
{
\filldraw[black] (\x/2,4) circle (4pt);
}
\foreach \x in {13,15,16,17,19,20,21}
{
\filldraw[gray] (\x/2,4) circle (4pt);
}
\foreach \x in {1,2,3}{
\foreach \y in {0,1,3,4,6,7}{
\filldraw[black] (-0.1-\x/10,\y) circle (0.5pt);
}
\filldraw[black] (11.1+\x/10,7) circle (0.5pt);
\filldraw[gray] (11.1+\x/10,6) circle (0.5pt);
\filldraw[gray] (11.1+\x/10,4) circle (0.5pt);
\filldraw[black] (11.1+\x/10,3) circle (0.5pt);
}
\node[] at (8.5,-0.3) {$\ell=1$};
\node[] at (10.5,3-0.3) {$\ell=1$};
\node[] at (10.5,6-0.3) {$\ell=1$};
\node[] at (8.5,1+0.4) {$\ell=0$};
\node[] at (10.5,4+0.4) {$\ell=0$};
\node[] at (10.5,7+0.4) {$\ell=0$};
\node[anchor=west] at (0,7+0.5) {\large Small $\tau$};
\node[anchor=west] at (0,4+0.5) {\large Large $\tau$};
\node[anchor=west] at (0,1+0.5) {\large Proposed $\tau$};
\end{tikzpicture}
\caption{Illustration of redundant samples for various choices of subsampling rate $\tau$. In each case, samples at level $\ell$ which are redundant for reaching a target threshold of $\varepsilon_{\ell}$ are shown in gray.}
\label{fig:subsamp_sketch}
\end{figure}
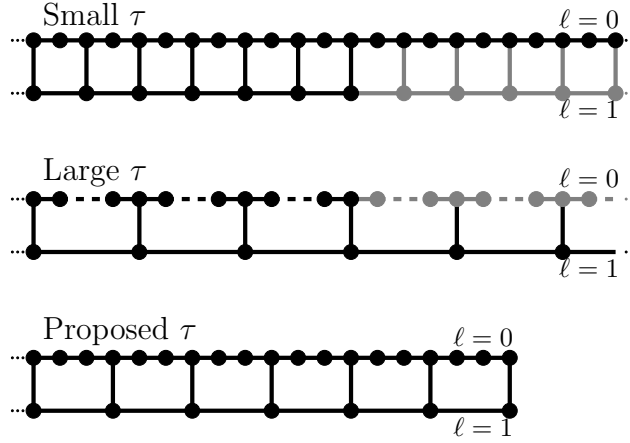

In the specific construction of the MLMCMC algorithm, one final refinement must be made to the expression for the cost per effective sample.  
For levels \(\ell > 0\), the total cost in Equation~\eqref{eq:effective_cost} includes the cost of generating a subchain of length \(\tau_{\ell}\) at level \(\ell-1\).  
However, the chain at level \(\ell-1\) has already been generated to a length of \(N_{\ell-1}^{\mathrm{eff}}\) effective samples.  
Provided that this coarse chain is sufficiently long, no additional computation is required on the coarse level to produce fine samples.  
Consequently, directly using \(\mathcal{C}_{\ell}^{\mathrm{eff}}\) in the total cost expression would double-count the lower-level work already accounted for in the coarser allocations.  
To avoid this, we define the additional cost per effective sample at level~\(\ell\), denoted \(\mathcal{C}_{\ell}^{\mathrm{add}}\), as  

\begin{equation}
\begin{aligned}
\mathcal{C}_{\ell}^{\text{add}}
&=
\mathcal{C}_{\ell}^{\text{eff}}
-
\left\lceil T_{\ell} \right\rceil
\tau_{\ell}\mathcal{C}_{\ell-1}^{\text{tot}}
\\
&= \left\lceil T_{\ell} \right\rceil \mathcal{C}_{\ell},
\end{aligned}
\end{equation}

with \(\mathcal{C}^{\text{add}}_0 = \mathcal{C}^{\text{eff}}_0\). The
corresponding effective sample sizes at all levels are then given by

\begin{equation}
N_{\ell}^{\text{eff}} = \dfrac{2}{\varepsilon^2}\left(\sum_{k=0}^L \sqrt{s_{k}^2\mathcal{C}_{k}^{\text{add}}}\right)\sqrt{\dfrac{s_{\ell}^2}{\mathcal{C}_{\ell}^{\text{add}}}}.
\label{eq:neff_optimal}
\end{equation}

Our proposed strategy is the following: rather than first specifying the
subsampling rates \(\tau_{\ell}\), computing the corresponding values of
\(N_{\ell}^{\mathrm{eff}}\) and then running the algorithm until each level
reaches its target effective sample size, we first calculate
\(N_{\ell}^{\mathrm{eff}}\) and choose the subsampling rates \(\tau_{\ell}\) so
that, under the estimated cost and autocorrelation assumptions, the chains are
expected to reach their target sample sizes at approximately the same time.
Since \(N_{\ell}^{\mathrm{eff}}\) denotes an effective sample size, the
corresponding target chain length is approximately
\(N_{\ell}=T_{\ell}N_{\ell}^{\mathrm{eff}}\), with integer rounding applied
afterwards. The unconstrained cost-based subsampling rate is then the ratio
between the target chain lengths

\begin{equation}
\tau_{\ell}^{\mathrm{cost}} = \left\lceil \dfrac{N_{\ell-1}}{N_{\ell}} \right\rceil = \left\lceil\dfrac{s_{\ell-1}T_{\ell-1}}{s_{\ell}T_{\ell}}\sqrt{\dfrac{\mathcal{C}_{\ell}^{\text{add}}}{\mathcal{C}_{\ell-1}^{\text{add}}}} \right\rceil.
\label{eq:cost_subsamp}
\end{equation}

This formulation is designed to satisfy the earlier assumption that the level
\(\ell-1\) chain is sufficiently long, so that no additional computations are
required at that level to generate fine samples beyond the rounded sample-count
adjustment below. When the perturbation diagnostic from
Subsection~\ref{sec:poisson_bias_diagnostic} is available, we combine this
cost-based rate with the empirical diagnostic condition in
Equation~\eqref{eq:empirical_bias_rate} by setting
\begin{equation}
\tau_{\ell} = \max\left\{\tau_{\ell}^{\mathrm{cost}},\tau_{\ell}^{\mathrm{diag}}\right\}.
\label{eq:proposed_subsamp}
\end{equation}
If no separate perturbation estimate is used during the adaptive run, one can
set \(\tau_{\ell}^{\mathrm{diag}}=1\) and perform the same check after the rate
has stabilized. In the numerical experiments below, the cost-based rate selected
by the adaptive procedure already satisfies the reported empirical check, so
the additional diagnostic does not change the observed values of \(\tau_{\ell}\).

To account for the fact that $\tau_{\ell}$ must be an integer, we round up to the nearest integer. For full completeness, the used sample size $N_{\ell}^{\text{pract}}$ at level $\ell-1$ should then also be increased to the nearest multiple of $N_{\ell}^{\text{pract}}$, so

\begin{equation}
N_{\ell-1}^{\text{pract}} = \left\lceil \dfrac{N_{\ell-1}}{N_{\ell}} \right\rceil N_{\ell}^{\text{pract}},
\end{equation}

with $N_L^{\text{pract}} = N_L$. In most cases, especially for long chains, this constitutes a negligible amount of additional computation. The sample sizes at all levels can then be constructed from this rule in descending order of $\ell$.

\begin{remark}
It is worth noting that the analysis presented above determines the subsampling
rate based on Equation~\eqref{eq:optimal_n}, which is derived under the
standard MLMC assumption of independent \(Y_{\ell}\). In practice, weak
correlations between levels often exist
\cite{peherstorfer_optimal_2016,peherstorfer_survey_2018,robbe_recycling_2018}.
While these inter-level covariances can in principle be incorporated into the
optimal sample allocation, their computation is nontrivial and computationally
expensive. Moreover, they are frequently found to be negligible compared to
within-level covariances~\cite{cui_multilevel_2024}. Thus, we omit these
cross-level covariance terms, accepting a minor suboptimality in exchange for a
substantially simpler and more efficient implementation.
\end{remark}

\subsection{The proposed adaptive algorithm}
\label{adaptive alg}
The use of our proposed subsampling rate in Equation~\eqref{eq:proposed_subsamp}
gives a practical adaptive procedure for the original MLMCMC algorithm. The
sample standard deviations \(s_{\ell}\), IATs \(T_{\ell}\), and costs per sample
\(\mathcal{C}_{\ell}\) are estimated from the samples generated so far, and the
resulting values of \(\tau_{\ell}\) are updated repeatedly during the run. If a
perturbation diagnostic is available, the same update also checks the empirical
diagnostic condition in Equation~\eqref{eq:empirical_bias_rate};
otherwise this condition can be evaluated after the cost-based rate has
stabilized.

Since \(\tau_{\ell}\) is updated using quantities estimated from the evolving
chains, this implementation is adaptive. The exact reset-kernel argument in
Subsection~\ref{justification} should therefore be interpreted as a reference
calculation for the fixed-\(\tau\) kernels visited by the adaptive procedure,
rather than as a full adaptive-MCMC convergence proof. Establishing such a proof
would require additional conditions, such as diminishing adaptation and
containment~\cite{haario_2001,roberts_rosenthal_2007,andrieu_thoms_2008}. In
the experiments below, the adaptive updates stabilize rapidly and the final
values are insensitive to the initial subsampling rate.

Equivalently, the adaptive phase can be viewed as a pilot stage for estimating
the fixed subsampling rates. After the estimates have stabilized, the rates may
be frozen and the fixed-\(\tau\) perturbation analysis above applies directly
to the production part of the run. The experiments reported below keep the same
practical implementation as the original MLMCMC algorithm, but the rapid
stabilization of the rates supports this fixed-rate interpretation.

The resulting adaptive algorithm is summarised in
Algorithm~\ref{alg:adapt_mlmcmc}.

\begin{algorithm}
\caption{MLMCMC with adaptive efficient subsampling}
\begin{algorithmic}
\State Define levels \(\ell=0,\ldots,L\). Sample \(\theta_L^0\sim\pi(\theta_L)\). For each level \(\ell\), define \(\theta_\ell^0\) as the truncation of \(\theta_L^0\) after the first \(M_\ell\) modes.
\State Choose initial rates \(\tau_1^{\mathrm{init}},\ldots,\tau_L^{\mathrm{init}}\), burn-in lengths \(B_0,\ldots,B_L\), recalculation interval \(\zeta\), and perturbation tolerances \(\varepsilon_{\ell}^{\mathrm{bias}}\) when a bias diagnostic is used.
\State Initialize \(\tau_{\ell}\leftarrow\tau_{\ell}^{\mathrm{init}}\) for \(\ell=1,\ldots,L\).
\For{\(n_0=0,\ldots,N_0-1\)}
\State Generate \(\theta_0^{n_0+1}\) using one Metropolis--Hastings step with pCN proposal.
\For{\(\ell=1,\ldots,L\)}
\If{enough new samples have been generated on level \(\ell-1\) according to the current \(\tau_{\ell}\)}
\State Generate \(\theta_\ell^{n_\ell+1}\) using one multilevel proposal step from Algorithm~\ref{alg:mlmcmc}.
\EndIf
\EndFor
\If{\(n_0\equiv 0 \pmod{\zeta}\) and \(n_0>B_0\)}
\For{\(\ell=0,\ldots,L\)}
\State Estimate the sample standard deviation \(s_{\ell}\), IAT \(T_{\ell}\), and additional level cost \(\mathcal{C}_{\ell}^{\mathrm{add}}\) from the samples currently available for \(Y_{\ell}\).
\If{\(\ell\geq1\)}
\State Set the cost-based rate
\[
\tau_{\ell}^{\mathrm{cost}}
=
\left\lceil
\dfrac{s_{\ell-1}T_{\ell-1}}{s_{\ell}T_{\ell}}
\sqrt{\dfrac{\mathcal{C}_{\ell}^{\mathrm{add}}}{\mathcal{C}_{\ell-1}^{\mathrm{add}}}}
\right\rceil .
\]
\State Set \(\tau_{\ell}^{\mathrm{diag}}\) from Equation~\eqref{eq:empirical_bias_rate} if available, and otherwise set \(\tau_{\ell}^{\mathrm{diag}}=1\).
\State Update \(\tau_{\ell}\leftarrow\max\{\tau_{\ell}^{\mathrm{cost}},\tau_{\ell}^{\mathrm{diag}}\}\).
\EndIf
\EndFor
\EndIf
\EndFor
\State \Return Markov chains of samples \(\{\theta_{\ell}^{n_{\ell}}\}_{n_{\ell}=0}^{N_{\ell}}\) at all levels \(\ell=0,\ldots,L\).
\end{algorithmic}
\label{alg:adapt_mlmcmc}
\end{algorithm}

\section{Numerics}
\label{sec:subsamp_numerics}

\begin{figure}
\centering
\includegraphics[width=0.8\textwidth]{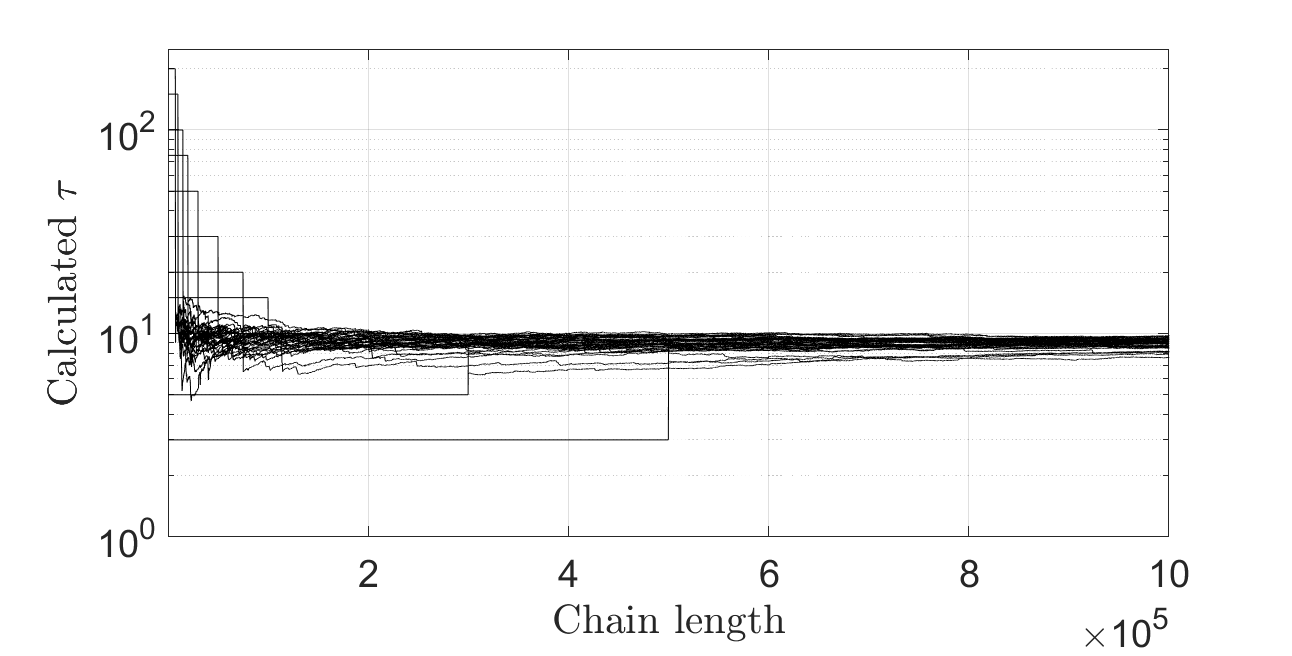}
\caption{Adaptive evolution of the estimate of the cost-based efficient subsampling rate from Equation~\eqref{eq:cost_subsamp} for increasing level-1 chain length. Evolution shown for all seven seeds and all 11 values of \(\tau^{\mathrm{init}}\).}
\label{fig:subsamp_estimate}
\end{figure}

We now illustrate the application of the proposed subsampling-rate selection
strategy in a numerical experiment for the elasticity beam problem. The setup
corresponds to the same model introduced in Section~\ref{ch:model setup}, and
in particular to the two-level test case used to obtain the numerical results
presented in Subsection~\ref{ch: effect subsampling}. The algorithm is executed
with 11 different initial subsampling rates, denoted \(\tau^{\mathrm{init}}\),
ranging from \(\tau^{\mathrm{init}} = 3\) to \(\tau^{\mathrm{init}} = 200\).
For each choice of \(\tau^{\mathrm{init}}\), the experiment is repeated using
7 independent random seeds to reduce stochastic variability.

The algorithm starts from the initial value of \(\tau\) and performs the first recalculation after \(1.5\times 10^6\) samples at level 0. The corresponding number of samples at level 1 is then \(1.5\times 10^6/\tau^{\mathrm{init}}\). After this point, the estimated value of \(\tau\) is updated repeatedly from the available estimates in Equation~\eqref{eq:cost_subsamp} and checked against the bias-budget condition in Equation~\eqref{eq:empirical_bias_rate}.

We first look at the adaptive estimate of \(\tau\). Figure~\ref{fig:subsamp_estimate} shows the evolution of the calculated value of \(\tau\) as the fine chain is constructed for all choices of seed and \(\tau^{\mathrm{init}}\). For visual clarity, the rounding in Equation~\eqref{eq:cost_subsamp} is not performed here.

We carry out the adaptive experiment for a total of \(10^6\) fine-level samples, including burn-in. We see in Figure~\ref{fig:subsamp_estimate} that immediately after the burn-in, the calculated values of \(\tau\) are already similar regardless of the chosen \(\tau^{\mathrm{init}}\). After \(10^6\) samples nearly all experiments converge to a rounded value of \(\tau = 9\) or \(10\), with no discernible dependence on the initial choice of \(\tau^{\mathrm{init}}\). The adaptive estimation thus provides a robust way of finding an efficient value of \(\tau\) regardless of the initial choice.

Based on this experiment, we therefore do not look for a globally optimal initial choice for the subsampling rate. The final estimate is insensitive to the initial value over the range tested here.

Next, we compare our Algorithm~\ref{alg:adapt_mlmcmc} to a run of standard MLMCMC with an IAT based subsampling rates to illustrate the computational gains from our algorithm. To this end, we now consider a 4-level experiment. The coarsest levels $\ell=0,1$ are defined in the same way as before. We additionally introduce two finer levels $\ell=2$ and $\ell=3$. The truncation numbers and grid resolutions of all levels are specified in Table~\ref{tab:fourlevel_values}.

\begin{table}[ht]
\centering
\begin{tabular}{r|cccc}
& $\ell = 0$ & $\ell = 1$ & $\ell = 2$ & $\ell = 3$ \\
\hline
Truncation \(M_{\ell}\) & 16 & 40 & 100 & 250 \\
Grid resolution (FEs) & $4\times 16$ & $8\times 32$ & $16\times 64$ & $32 \times 128$
\end{tabular}
\caption{Truncation numbers of the KL expansion and grid resolution in number of finite elements as height $\times$ width for all levels of the 4-level experiment.}
\label{tab:fourlevel_values}
\end{table}

We initialize our Algorithm~\ref{alg:adapt_mlmcmc} with subsampling rates $\tau_1^{\text{init}} = 100$ and $\tau_2^{\text{init}}=\tau_3^{\text{init}} = 5$. For the standard MLMCMC algorithm, the subsampling rates are initialized as $\tau_1 = 1000$ and $\tau_2=\tau_3 = 5$. After the burn-in phase, we calculate the IATs at all levels and replace the initial choices in the standard MLMCMC algorithm by the fixed values $\tau_{\ell} = \lceil T_{\ell-1} \rceil$. The IAT values are shown in Table~\ref{tab:iat_values}.

\begin{table}[ht]
\centering
\begin{tabular}{r|cccc}
& $\ell = 0$ & $\ell = 1$ & $\ell = 2$ & $\ell = 3$ \\
\hline
IAT & 1188 & 49.5 & 36.0 & 6.20
\end{tabular}
\caption{IAT values at all levels of the 4-level experiment using fixed subsampling rates.}
\label{tab:iat_values}
\end{table}

\begin{figure}
\centering
\includegraphics[width=\textwidth]{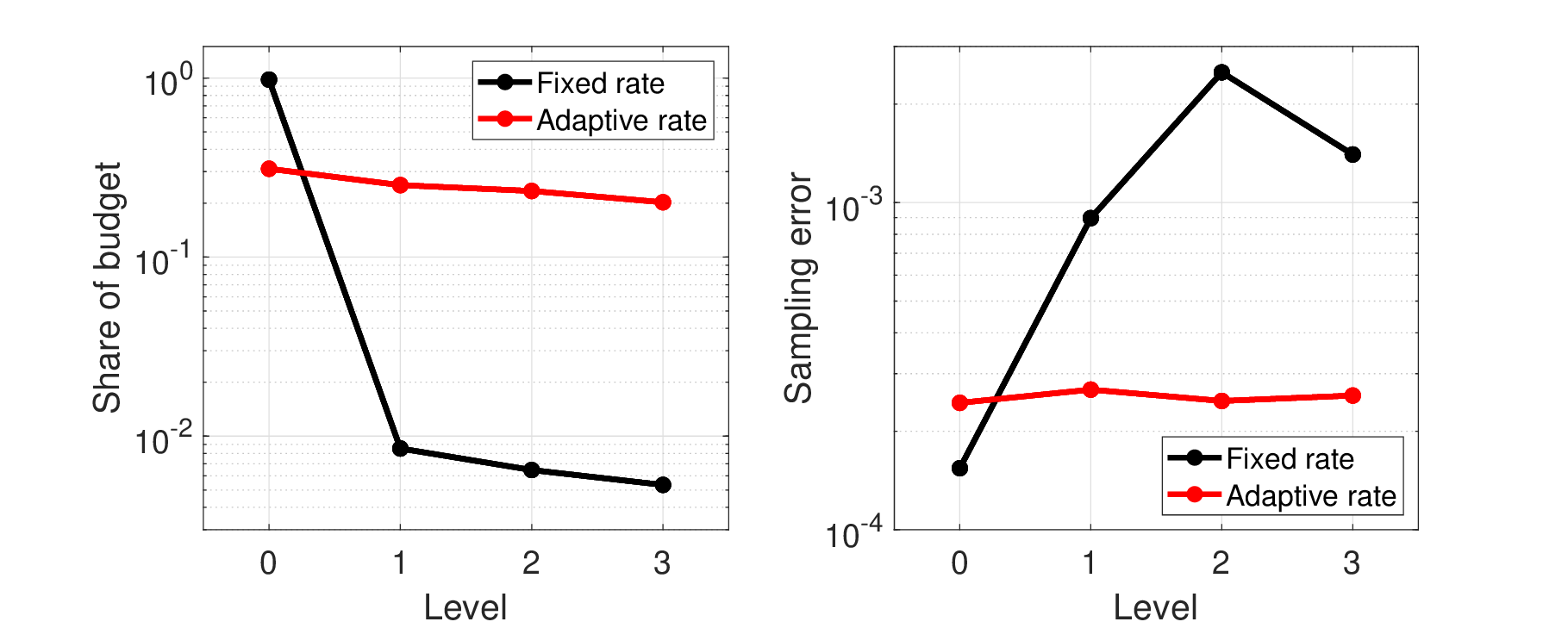}
\caption{Left: relative share of total computational budget spent at each level of the multilevel algorithm for fixed choice of $\tau$ and $\tau$ obtained from Algorithm~\ref{alg:adapt_mlmcmc}. Right: sampling error $\varepsilon$ of estimator at each level for both approaches with total runtime of $10^4$ seconds.}
\label{fig:budgeted_errs}
\end{figure}

The chains obtained in this experiment exhibit strong autocorrelation.
Note that, although the IAT decreases with increasing level $\ell$, it remains substantially greater than one.
When the subsampling rate is chosen as \(\tau_{\ell} = \lceil T_{\ell-1} \rceil\), the large values of \(T_1\) and \(T_2\) considerably amplify the overall computational cost, since generating a single finest-level sample requires a number of coarse samples equal to the product of all \(\tau_{\ell'+1}\) for \(\ell' < \ell\). With the values of \(T_{\ell}\) found here, each sample at the finest level \(\ell=3\) requires over \(2\cdot10^6\) samples at level 0; this cost is further multiplied by approximately 6 when expressed per effective finest-level sample.

The differing subsampling rates used by the two algorithms lead to different distributions of computational effort across the levels. This effect is illustrated in the left plot of Figure~\ref{fig:budgeted_errs}, which shows the relative proportion of the total computational budget allocated to each level of the multilevel algorithm for both approaches.

The subsampling rate defined in Equation~\eqref{eq:proposed_subsamp} yields a well-balanced distribution of computational resources across all levels of the MLMCMC algorithm. This contrasts sharply with the standard MLMCMC algorithm, which employs a fixed choice of \(\tau_{\ell} = \lceil T_{\ell-1} \rceil\). Owing to the strong autocorrelation at the coarse levels, this fixed-rate strategy causes the majority of the computational budget to be spent on the coarsest level, which accounts for approximately 99\% of the total runtime.

The differing allocations of computational effort between the two algorithms lead to pronounced differences in the level-wise sampling errors \(\varepsilon_{\ell}\). This is illustrated in the right plot of Figure~\ref{fig:budgeted_errs} where both forms of the algorithm were executed for a total of $10^4$ seconds to calculate the respective values of $\varepsilon_{\ell}$ after a fixed time. The adaptive subsampling rate from Equation~\eqref{eq:proposed_subsamp} results in very similar values at all levels. The fixed choice of $\tau_{\ell} = \lceil T_{\ell-1} \rceil$ spends most of its time at the coarsest level and as such manages to obtain a smaller value of $\varepsilon_0$. However, this is obtained to the detriment of the values of $\varepsilon_{\ell}$ for $\ell \geq 1$ which are considerably larger due to the lower sample sizes at these levels.

\begin{figure}[ht]
\centering
\includegraphics[width=0.75\textwidth]{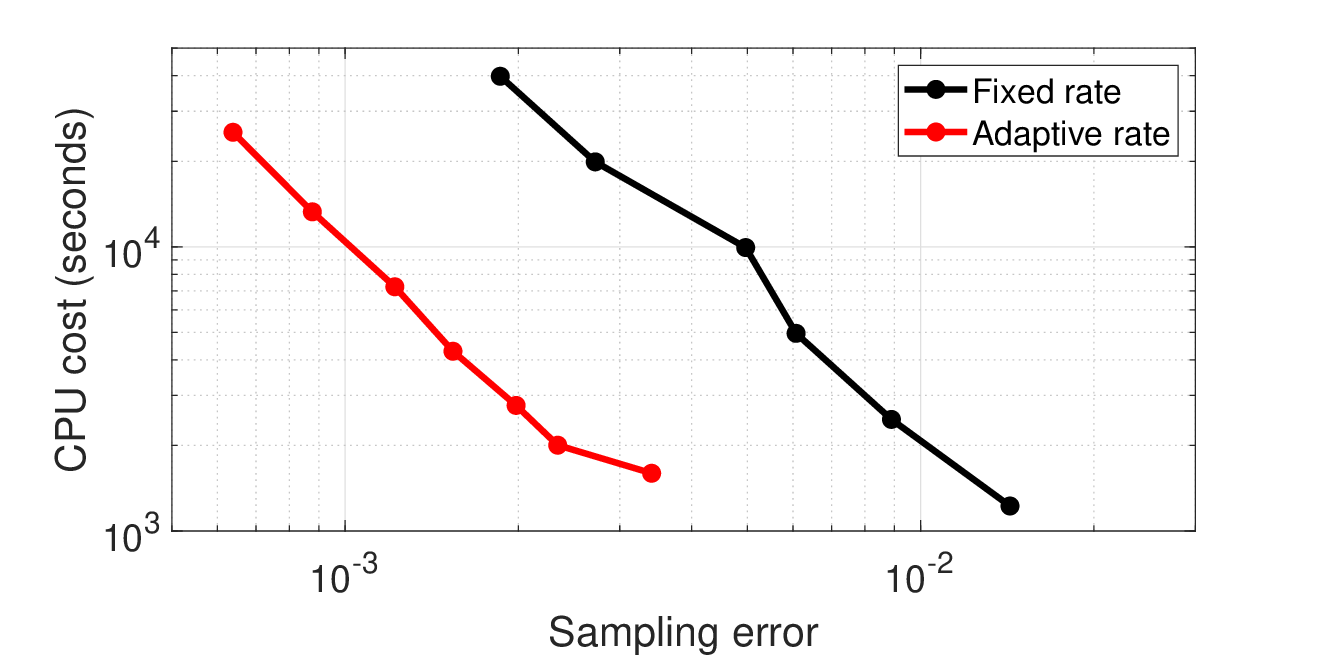}
\caption{CPU cost to achieve prescribed sampling error $\varepsilon$ for both subsampling approaches of the MLMCMC algorithm.}
\label{fig:error_cost}
\end{figure}

The increased errors at finer levels, $\ell \geq 1$, lead to a larger overall sampling error. This is shown in Figure~\ref{fig:error_cost}, where the total sampling error is plotted against the computational budget required.

Both methods follow the asymptotic convergence rate obtained by \cite{dodwell_hierarchical_2015}. However, the constant of proportionality is significantly different in these runs, with the proposed adaptive subsampling rate requiring approximately an order of magnitude less computation time over the tested error range. These numerical experiments support the conclusion that significant gains can be made from a judicious choice of subsampling rate in this highly autocorrelated benchmark. The benefits are most pronounced for highly autocorrelated chains, while for problems with lower autocorrelation the improvement is expected to be less dramatic.
The construction based on Equation~\eqref{eq:neff_optimal} should therefore be read as an approximate cost--error criterion under the estimated variance, IAT, and cost assumptions, not as a proof of global cost optimality.

\section{Conclusion}
\label{ch:conclusion}
The MLMCMC algorithm is widely used for Bayesian inverse problems, but its efficiency depends strongly on how often coarse-level states are passed to the next finer level. In the original IAT-based rule, the subsampling rate is chosen using the IAT as a proxy for proposal decorrelation. For highly autocorrelated coarse chains this can require many coarse updates for each fine-level proposal. The resulting coarse-level sampling error is then much smaller than the fine-level contributions, and a substantial part of the computation is spent reducing an error term that is no longer dominant.

This paper argues that the subsampling rate should instead be treated as a
cost--error allocation parameter. The theoretical discussion shows that, in an
ideal reset-reference setting, approximate independence is not essential for the
fine-level Metropolis--Hastings calculation: reversibility of the coarse
proposal with respect to the coarse posterior is sufficient. This observation
does not make the practical continued-chain implementation into a new exact
sampler. It only provides a useful reference against which the finite-subsampling
effect in the practical chain can be viewed as a perturbation.

The Poisson-equation identity gives an expression for this perturbation bias,
and a geometric bound follows conditionally under a state-averaged lower-level
forgetting assumption. In computation, the signed-Poisson calculation is used as
a heuristic empirical error-budget diagnostic. In the elasticity experiments
this finite-subsampling effect is visible at the selected rates: it is not
below the paired level-one Monte Carlo error, but it remains below the
discretization proxy used in the reported multilevel error budget.

The resulting subsampling-rate criterion balances the variance,
autocorrelation, and cost contributions across levels, with the aim that the
level corrections reach their target sampling errors on comparable time scales.
Because the required sample variances \(s_{\ell}^2\) and integrated
autocorrelation times \(T_{\ell}\) depend on the chosen rate, we implemented the
criterion adaptively during the simulation. The adaptive phase can also be
interpreted as a pilot stage after which the stabilized rates are fixed. The
reported elasticity experiments demonstrate rapid stabilization near
\(\tau\approx9\)--10 and substantial computational savings in a highly
autocorrelated benchmark. The Darcy-flow experiment gives supporting evidence
of qualitatively similar finite-subsampling behaviour in another elliptic
inverse problem, while the scope of the evidence remains numerical rather than
model-independent. Overall, the results support replacing the IAT-based
decorrelation rule by a cost--error tuning criterion, accompanied by explicit
empirical perturbation diagnostics when reduced subsampling rates are used.

\bibliographystyle{siamplain}
\bibliography{references}
\end{document}

%% file: ex_shared.tex
\usepackage{lipsum}
\usepackage{amsfonts}
\usepackage{graphicx}
\usepackage{epstopdf}
\usepackage{ifthen}
\usepackage{algorithm}      
\usepackage{algpseudocode}
\algrenewcommand\alglinenumber[1]{}  

\ifpdf
  \DeclareGraphicsExtensions{.eps,.pdf,.png,.jpg}
\else
  \DeclareGraphicsExtensions{.eps}
\fi

\usepackage{tikz}
\usetikzlibrary{arrows.meta,positioning,calc}

\newsiamremark{remark}{Remark}
\newsiamremark{hypothesis}{Hypothesis}
\crefname{hypothesis}{Hypothesis}{Hypotheses}
\newsiamthm{claim}{Claim}

\headers{Efficient subsampling-rate selection for practical MLMCMC}{L. Guilliams, P. Vanmechelen, and G. Samaey}

\title{Efficient subsampling-rate selection for practical multilevel Markov chain Monte Carlo\thanks{\mbox{}
\funding{P. Vanmechelen is an SB Ph.D. fellow of the Research Foundation Flanders (FWO) under grant 1SD1825N. Additional funding was provided by the KU Leuven Research Fund under grant C14/23/098.}}}

\author{Lise Guilliams\thanks{NUMA, Department of Computer Science KU Leuven, Celestijnenlaan 200A, B-3001 Leuven}
\and Pieter Vanmechelen \footnotemark[2]
\and Giovanni Samaey \footnotemark[2]}

\usepackage{amsopn}

\makeatletter
\newcommand*{\addFileDependency}[1]{
  \typeout{(#1)}
  \@addtofilelist{#1}
  \IfFileExists{#1}{}{\typeout{No file #1.}}
}
\makeatother


%% file: ex_article.bbl
\begin{thebibliography}{10}

\bibitem{alquier_noisy_2016}
{\sc P.~Alquier, N.~Friel, R.~Everitt, and A.~Boland}, {\em Noisy {Monte}
  {Carlo}: Convergence of {Markov} chains with approximate transition kernels},
  Statistics and Computing, 26 (2016), pp.~29--47,
  \url{https://doi.org/10.1007/s11222-014-9521-x}.

\bibitem{andrieu_thoms_2008}
{\sc C.~Andrieu and J.~Thoms}, {\em A tutorial on adaptive {MCMC}}, Statistics
  and Computing, 18 (2008), pp.~343--373,
  \url{https://doi.org/10.1007/s11222-008-9110-y}.

\bibitem{baecher_reliability_2003}
{\sc G.~Baecher and J.~Christian}, {\em Reliability and {Statistics} in
  {Geotechnical} {Engineering}}, Wiley, 2003.

\bibitem{Bia_as_2023}
{\sc P.~Białas, P.~Korcyl, and T.~Stebel}, {\em Analysis of autocorrelation
  times in neural markov chain monte carlo simulations}, Physical Review E, 107
  (2023), \url{https://doi.org/10.1103/physreve.107.015303}.

\bibitem{charrier_strong_2012}
{\sc J.~Charrier}, {\em Strong and {Weak} {Error} {Estimates} for {Elliptic}
  {Partial} {Differential} {Equations} with {Random} {Coefficients}}, SIAM
  Journal on Numerical Analysis, 50 (2012), pp.~216--246,
  \url{https://doi.org/10.1137/100800531}.

\bibitem{christen_markov_2005}
{\sc J.~A. Christen and C.~Fox}, {\em Markov chain {Monte} {Carlo} {Using} an
  {Approximation}}, Journal of Computational and Graphical Statistics, 14
  (2005), pp.~795--810, \url{https://doi.org/10.1198/106186005X76983}.

\bibitem{chung2020introduction}
{\sc M.~K. Chung}, {\em Introduction to random fields}, arXiv preprint
  arXiv:2007.09660,  (2020).

\bibitem{cliffe_multilevel_2011}
{\sc K.~A. Cliffe, M.~B. Giles, R.~Scheichl, and A.~L. Teckentrup}, {\em
  Multilevel {Monte} {Carlo} methods and applications to elliptic {PDEs} with
  random coefficients}, Computing and Visualization in Science, 14 (2011),
  pp.~3--15, \url{https://doi.org/10.1007/s00791-011-0160-x}.

\bibitem{cotter_mcmc_2013}
{\sc S.~Cotter, G.~Roberts, A.~M. Stuart, and D.~White}, {\em Mcmc methods for
  functions: {Modifying} old algorithms to make them faster}, Statistical
  Science, 28 (2013), pp.~424--446, \url{https://doi.org/10.1214/13-STS421}.

\bibitem{cui_multilevel_2024}
{\sc T.~Cui, G.~Detommaso, and R.~Scheichl}, {\em Multilevel
  dimension-independent likelihood-informed {MCMC} for large-scale inverse
  problems}, Inverse Problems, 40 (2024), p.~035005,
  \url{https://doi.org/10.1088/1361-6420/ad1e2c}.

\bibitem{dodwell_hierarchical_2015}
{\sc T.~J. Dodwell, C.~Ketelsen, R.~Scheichl, and A.~L. Teckentrup}, {\em A
  {Hierarchical} {Multilevel} {Markov} {Chain} {Monte} {Carlo} {Algorithm} with
  {Applications} to {Uncertainty} {Quantification} in {Subsurface} {Flow}},
  SIAM/ASA Journal on Uncertainty Quantification, 3 (2015), pp.~1075--1108,
  \url{https://doi.org/10.1137/130915005}.

\bibitem{ghanem2003stochastic}
{\sc R.~G. Ghanem and P.~D. Spanos}, {\em Stochastic Finite Elements: A
  Spectral Approach}, Dover Publications, 2003.

\bibitem{giles_multilevel_2008}
{\sc M.~B. Giles}, {\em Multilevel {Monte} {Carlo} {Path} {Simulation}},
  Operations Research, 56 (2008), pp.~607--617,
  \url{https://doi.org/10.1287/opre.1070.0496}.

\bibitem{glynn_meyn_1996}
{\sc P.~W. Glynn and S.~P. Meyn}, {\em A {Liapounov} bound for solutions of the
  {Poisson} equation}, The Annals of Probability, 24 (1996), pp.~916--931,
  \url{https://doi.org/10.1214/aop/1039639370}.

\bibitem{haario_2001}
{\sc H.~Haario, E.~Saksman, and J.~Tamminen}, {\em An adaptive metropolis
  algorithm}, Bernoulli, 7 (2001), pp.~223--242,
  \url{https://doi.org/10.2307/3318737}.

\bibitem{hastings1970}
{\sc W.~Hastings}, {\em Monte carlo sampling methods using markov chains and
  their application}, Biometrika, 57 (1970), pp.~97--109,
  \url{https://doi.org/10.1093/biomet/57.1.97}.

\bibitem{iqbal2015fundamentals}
{\sc G.~M. Iqbal and A.~Satter}, {\em Fundamentals of Fluid Flow through Porous
  Media}, Reservoir Engineering, Elsevier Science \& Technology, United States,
  2015.

\bibitem{law_data_2015}
{\sc K.~Law, A.~Stuart, and K.~Zygalakis}, {\em Data {Assimilation}: {A}
  {Mathematical} {Introduction}}, vol.~62 of Texts in {Applied} {Mathematics},
  Springer International Publishing, Cham, 2015,
  \url{https://doi.org/10.1007/978-3-319-20325-6}.

\bibitem{levin_peres_wilmer_2017}
{\sc D.~A. Levin, Y.~Peres, and E.~L. Wilmer}, {\em Markov Chains and Mixing
  Times}, American Mathematical Society, 2~ed., 2017,
  \url{https://pages.uoregon.edu/dlevin/MARKOV/mcmt2e.pdf}.

\bibitem{link_thinning_2012}
{\sc W.~A. Link and M.~J. Eaton}, {\em On thinning of chains in {MCMC}},
  Methods in Ecology and Evolution, 3 (2012), pp.~112--115,
  \url{https://doi.org/10.1111/j.2041-210X.2011.00131.x}.

\bibitem{lykkegaard_multilevel_2023}
{\sc M.~B. Lykkegaard, T.~J. Dodwell, C.~Fox, G.~Mingas, and R.~Scheichl}, {\em
  Multilevel {Delayed} {Acceptance} {MCMC}}, SIAM/ASA Journal on Uncertainty
  Quantification, 11 (2023), pp.~1--30,
  \url{https://doi.org/10.1137/22M1476770}.

\bibitem{madrigalcianci2021analysisclassmultilevelmarkov}
{\sc J.~P. Madrigal-Cianci, F.~Nobile, and R.~Tempone}, {\em Analysis of a
  class of multilevel markov chain monte carlo algorithms based on independent
  metropolis--hastings}, SIAM/ASA Journal on Uncertainty Quantification, 11
  (2023), pp.~91--138, \url{https://doi.org/10.1137/21M1420927}.

\bibitem{metropolis1953}
{\sc N.~Metropolis, A.~Rosenbluth, M.~Rosenbluth, A.~Teller, and E.~Teller},
  {\em Equations of state calculations by fast computing machines}, Journal of
  Chemical Physics, 21 (1953), pp.~1087--1092,
  \url{https://doi.org/10.1063/1.1699114}.

\bibitem{mijatovic_vogrinc_2018}
{\sc A.~Mijatovi{\'c} and J.~Vogrinc}, {\em On the {Poisson} equation for
  {Metropolis--Hastings} chains}, Bernoulli, 24 (2018), pp.~2401--2428,
  \url{https://doi.org/10.3150/17-BEJ932}.

\bibitem{mitrophanov_2005}
{\sc A.~Y. Mitrophanov}, {\em Sensitivity and convergence of uniformly ergodic
  {Markov} chains}, Journal of Applied Probability, 42 (2005), pp.~1003--1014,
  \url{https://doi.org/10.1239/jap/1134587812}.

\bibitem{peherstorfer_optimal_2016}
{\sc B.~Peherstorfer, K.~Willcox, and M.~Gunzburger}, {\em Optimal {Model}
  {Management} for {Multifidelity} {Monte} {Carlo} {Estimation}}, SIAM Journal
  on Scientific Computing, 38 (2016), pp.~A3163--A3194,
  \url{https://doi.org/10.1137/15M1046472}.

\bibitem{peherstorfer_survey_2018}
{\sc B.~Peherstorfer, K.~Willcox, and M.~Gunzburger}, {\em Survey of
  {Multifidelity} {Methods} in {Uncertainty} {Propagation}, {Inference}, and
  {Optimization}}, SIAM Review, 60 (2018), pp.~550--591,
  \url{https://doi.org/10.1137/16M1082469}.

\bibitem{cotter_probabilistic_2015}
{\sc S.~Reich and C.~Cotter}, {\em Probabilistic Forecasting and Bayesian Data
  Assimilation}, Cambridge University Press, 2015,
  \url{https://doi.org/10.1017/CBO9781107706804}.

\bibitem{robbe_recycling_2018}
{\sc P.~Robbe, D.~Nuyens, and S.~Vandewalle}, {\em Recycling {Samples} in the
  {Multigrid} {Multilevel} ({Quasi}-){Monte} {Carlo} {Method}}, June 2018,
  \url{https://doi.org/10.48550/arXiv.1806.05619}.
\newblock arXiv:1806.05619 [math].

\bibitem{roberts_rosenthal_2007}
{\sc G.~O. Roberts and J.~S. Rosenthal}, {\em Coupling and ergodicity of
  adaptive markov chain monte carlo algorithms}, Journal of Applied
  Probability, 44 (2007), pp.~458--475,
  \url{https://doi.org/10.1239/jap/1183667414}.

\bibitem{seiffert2024estimatingeffectivesamplesize}
{\sc L.~Seiffert and F.~Pereira}, {\em Estimating the effective sample size for
  an inverse problem in subsurface flows}, arXiv preprint arXiv:2408.13411,
  (2024).

\bibitem{smith_uncertainty_2024}
{\sc R.~Smith}, {\em Uncertainty {Quantification}: theory, implementation and
  applications}, Society for Industrial and Applied Mathematics, Philadelphia,
  2024.

\bibitem{tarantola_inverse_2005}
{\sc A.~Tarantola}, {\em Inverse {Problem} {Theory} and {Methods} for {Model}
  {Parameter} {Estimation}}, Society for Industrial and Applied Mathematics,
  Jan. 2005, \url{https://doi.org/10.1137/1.9780898717921}.

\bibitem{tierney_1994}
{\sc L.~Tierney}, {\em Markov chains for exploring posterior distributions},
  The Annals of Statistics, 22 (1994), pp.~1701--1728,
  \url{https://doi.org/10.1214/aos/1176325750}.

\bibitem{vanmechelen_multilevel_2025}
{\sc P.~Vanmechelen, G.~Lombaert, and G.~Samaey}, {\em Multilevel {Markov}
  {Chain} {Monte} {Carlo} with {Likelihood} {Scaling} for {Bayesian}
  {Inversion} with {High}-resolution {Observations}}, SIAM/ASA Journal on
  Uncertainty Quantification, 13 (2025), pp.~375--399,
  \url{https://doi.org/10.1137/24M1634333}.

\bibitem{wicehardware}
{\sc {Vlaams Supercomputing Center}}, {\em {wICE} hardware}, 2025.
\newblock
  \nolinkurl{https://docs.vscentrum.be/leuven/tier2_hardware/wice_hardware.html}
  [Accessed: September 22, 2026].

\bibitem{zheng2017simulation}
{\sc Z.~Zheng and H.~Dai}, {\em Simulation of multi-dimensional random fields
  by karhunen--lo{\`e}ve expansion}, Computer Methods in Applied Mechanics and
  Engineering, 324 (2017), pp.~221--247,
  \url{https://doi.org/10.1016/j.cma.2017.05.018}.

\end{thebibliography}
